\documentclass[12pt]{article}
\usepackage{fullpage,comment,authblk,stackrel}
\usepackage[small]{caption}
\usepackage[pagewise]{lineno}
\usepackage{blindtext}
\usepackage{tikz-cd}

\usepackage{hyperref}

\usepackage{nicefrac,color}
\usepackage{hyperref}
\hypersetup
{
    colorlinks,	%
    citecolor=blue,%
    filecolor=blue,%
    linkcolor=blue,%
    urlcolor=blue
}
\usepackage{amssymb,mathtools,amsthm}
\usepackage[all,cmtip]{xy}
\usepackage{eulervm,xfrac,mathdots}
\usepackage{rotating, setspace}

\newcommand{\Rspace}        	{{\mathbb R}}
\newcommand{\Zspace}        	{{\mathbb Z}}
\newcommand{\Sspace}        	{{\mathbb S}}

\newcommand{\st}			{\mathsf{st}}

\newcommand{\cl}			{\mathsf{cl}}

\newcommand{\epi}			{\twoheadrightarrow}
\newcommand{\mono}		{\hookrightarrow}

\newcommand{\Open}         	{\mathsf{Open}}
\newcommand{\COpen}         	{\mathsf{COpen}}

\newcommand{\Ab}          		{\mathsf{Ab}}

\newcommand{\Sheaf}          	{{\mathsf{Sh}}}

\newcommand{\Loc}          	{{\mathsf{Loc}}}

\newcommand{\Cosheaf}          	{{\mathsf{Cosh}}}

\newcommand{\Coloc}          	{{\mathsf{Coloc}}}

\newcommand{\Bisheaf}          	{{\mathsf{Bish}}}

\newcommand{\Stack}          	{{\mathsf{Stack}}}
\newcommand{\Ent}          	{{\mathsf{Ent}}}
\newcommand{\Exit}          	{{\mathsf{Exit}}}

\newcommand{\Etal}         		{\mathsf{Etale}}

\newcommand{\Basic}         	{\mathsf{Basic}}

\newcommand{\Acat}          	{{\mathsf{A}}}

\newcommand{\Pcat}          	{{\mathsf{P}}}
\newcommand{\Ucat}          	{{\mathsf{U}}}

\newcommand{\Hgroup}          	{{\mathsf{H}}}

\newcommand{\Sstrat}          	{{\mathcal{S}}}

\newcommand{\Tstrat}          	{{\mathcal{T}}}
\newcommand{\Kstrat}          	{{\mathcal{K}}}
\newcommand{\Lstrat}          	{{\mathcal{L}}}

\newcommand{\Jcontrol}          	{{\mathcal{J}}}

\newcommand\cov[1]          	{{\mathsf{#1}}}

\newcommand{\cosheaf}[1]          	{\underline{\mathsf{{#1}}}}
\newcommand{\sheaf}[1]          		{\overline{\mathsf{{#1}}}}
\newcommand{\bisheaf}[1]          	{\underline{\overline{\mathsf{{#1}}}}}

\newcommand{\Fstack}          	{\ulbar {\mathcal{F}}}
\newcommand{\uFstack}          	{\ubar {\mathcal{F}}}
\newcommand{\lFstack}          	{\lbar {\mathcal{F}}}

\newcommand{\Gstack}          	{\ulbar {\mathcal{G}}}
\newcommand{\uGstack}          	{\ubar {\mathcal{G}}}
\newcommand{\lGstack}          	{\lbar {\mathcal{G}}}
\newcommand{\Hstack}          	{\ulbar {\mathcal{H}}}

\newcommand{\Epi}			{{\mathsf{Epi}}}	
\newcommand{\Mono}		{{\mathsf{Mono}}}
\newcommand{\Iso}         	 	{{\mathsf{Iso}}}

\newcommand{\Ffunc}          	{{\mathsf{F}}}
\newcommand{\Gfunc}          	{{\mathsf{G}}}
\newcommand{\Hfunc}          	{{\mathsf{H}}}
\newcommand{\Ifunc}          	{{\mathsf{I}}}

\newcommand{\Qfunc}          	{{\mathsf{Q}}}

\newcommand\ubar[1]		{{\overline{#1}}}	
\newcommand\lbar[1]		{{\underline{#1}}}	
\newcommand\ulbar[1]		{\overline{{\underline{#1}}}}

\newcommand{\colim}		{\mathsf{colim}\;}
\newcommand{\mylim}		{\mathsf{lim}\;}

\newcommand{\image}		{\mathsf{im}\;}
\newcommand{\coimage}		{\mathsf{coim}\;}

\newcommand{\id}			{\mathrm{id}}

\newcommand{\Dist}			{\mathsf{dist}}

\newcommand\define[1]		{{\bf{#1}}}	

\newcommand{\Haus}		{\mathsf{Haus}}

\makeatletter
\def\moverlay{\mathpalette\mov@rlay}
\def\mov@rlay#1#2{\leavevmode\vtop{%
   \baselineskip\z@skip \lineskiplimit-\maxdimen
   \ialign{\hfil$\m@th#1##$\hfil\cr#2\crcr}}}
\newcommand{\charfusion}[3][\mathord]{
    #1{\ifx#1\mathop\vphantom{#2}\fi
        \mathpalette\mov@rlay{#2\cr#3}
      }
    \ifx#1\mathop\expandafter\displaylimits\fi}
\makeatother

\newtheoremstyle{amit}
{7pt}
{7pt}
{}
{7pt}
{\bf}
{:}
{.5em}
{}

\begin{document}

\title{Persistent Local Systems}
\author[1]{Robert MacPherson}
\author[2]{Amit Patel \thanks{This research was partially supported by NSF grant CCF 1717159} }
\affil[1]{School of Mathematics, Institute for Advanced Study}
\affil[2]{Department of Mathematics, Colorado State University}
\date{}
\renewcommand\Authands{ and }

\maketitle

\abstract{
In this paper, we give lower bounds for the homology of the fibers of a map to a manifold. 
Using new sheaf theoretic methods, we show that these lower bounds persist over whole open sets of the manifold
and that they are stable under perturbations of the map.
This generalizes certain ideas of persistent homology to higher dimensions.
}

\theoremstyle{amit}
\newtheorem{defn}{Definition}[section]
\newtheorem{prop}[defn]{Proposition}
\newtheorem{lem}[defn]{Lemma}
\newtheorem{thm}[defn]{Theorem}
\newtheorem{corr}[defn]{Corollary}
\newtheorem{rmk}[defn]{Remark}
\newtheorem{ex}[defn]{Example}

\section{Introduction}
\label{sec:introduction}

In this paper, we consider a continuous mapping $f: X \to M$ of a topological space $X$ to a manifold $M$.  
We think of $f$ as a family of fibers $f^{-1}p \subseteq X$ parameterized by points $p$ in $M$. 
We are interested in topological properties of these fibers that are stable under small perturbations of the map $f$.  
Besides being of mathematical interest in its own right, this stability requirement is important for applications, where $f$ may be subject to perturbations from measurement noise or computational error.   
Stability is particularly appealing for data analysis as data is inherently noisy. 

We study the homology groups of the fibers $\Hfunc_j( f^{-1}p)$ and their dimensions, the Betti numbers 
$\beta_j( f^{-1}p)$.  
We are mainly interested in lower bounds on the Betti numbers that continue to hold for small perturbations of $f$.  
Lower bounds are important because linearly independent elements of $\Hfunc_j( f^{-1}p)$ 
that remain linearly independent under small perturbations are regarded as interesting features of the family.  
The stability requirement is a serious one:  
even when $\beta_j ( f^{-1}p)$ is large, there can exist perturbations $\tilde f$ arbitrarily close to $f$ such that 
$\beta_j ( \tilde f^{-1} p )=0$.

\paragraph{Conventions.}  In this introduction,    we fix a map $f:X\to M$, where $M$ is a manifold, a metric $d$ on $M$, and an orientation of $M$. All homology groups are with field coefficients and of fixed degree $j$ and all open sets are path-connected.

\paragraph{Persistent Dimension.} The simplest statement of the type of result in this paper is the following.
For every open set $U\subseteq M$, we will associate, in a few paragraphs below, 
a nonnegative integer $\mathcal P(U)$ called the {\em persistent dimension} of $U$ which
has the following properties:

\begin{enumerate}
\item{Betti number lower bound}: $\beta_j(f^{-1}p)\geq \mathcal P(U)$ for all $p\in U$.
That is, $\mathcal P(U)$ is a lower bound for the Betti numbers of all the fibers over $U$. 
\label{prop:dimension_one}
\item {Stability}: For all sufficiently small $\epsilon \geq 0$ and for all perturbed $\tilde f$ that is $\epsilon$-close 
to $f$, we have 
$\beta_j (\tilde f^{-1}p) \geq \mathcal P(U)$ for all $p\in U$ that is at least $\epsilon$ away from the boundary of $U$. 
In other words, $\mathcal P(U)$ is still a lower bound for the Betti numbers of the fibers if $U$ is shrunk by~$\epsilon$.
\label{prop:dimension_two}   
\end{enumerate}
The metric in which we ask $\tilde f$ to be $\epsilon$-close to $f$ is $\sup_{x \in X} d(f x,\tilde f x)$. 
By the second property, it follows that 
for all $p\in U$, there is an $\epsilon$ so that for all $\tilde f$ that is $\epsilon$ close to~$f$, 
$\beta_j(\tilde f^{-1}p)\geq\mathcal P(U)$.
In other words, the lower bounds on Betti numbers provided by $\mathcal P(U)$ are meaningful in the presence of small enough error in the determination of~$f$.

We would like to say that the $\mathcal P(U)$-dimension part of $\Hfunc_j(f^{-1}p)$ guaranteed by the first property 
forms a family over $U$. 
To do that, we need to recall the idea of a local system. 

\paragraph{Local Systems.} A {local system} $\mathcal{L}$ over a space $U$, also called a locally constant sheaf over $U$, is  a ``family'' of vector spaces parameterized by points in $U$.  It may be defined as the following data:  
	\begin{enumerate}
	\item a vector space $\mathcal L_p$ for every point $p\in U$ called the stalk of $\mathcal L$ at $p$, and 
	\item an isomorphism $\mathcal L_\gamma: \mathcal L_p\to \mathcal L_q$ for every homotopy class 
	$\gamma$ of paths from $p$ to $q$ called the monodromy along $\gamma$.
	\end{enumerate}
Local systems over $U$ form a category: morphisms $\mathcal L \to \mathcal L'$ are sets of 
linear maps $\mathcal L_p\to \mathcal L'_p$ for each $p\in U$ that commute with the monodromy maps.
The isomorphisms $\mathcal L_\gamma$ are required to be compatible with composition of paths.
In other words, a local system is a functor from the fundamental groupoid of $U$ to the category of vector spaces,
and a morphism of local systems is a natural transformation of functors.
If $U'$ is a subset of $U$, a local system $\mathcal L$ over $U$ restricts to a local system $\mathcal L|_{U'}$ over $U'$ by throwing away all the data that does not lie in $U'$.
If $U$ is path-connected, the vector spaces $\mathcal L_p$ all have the same dimension and if further $U$ 
is simply connected, then they may all be identified with a single vector space $V$ 
so that all the maps $L_\gamma$ are the identity on $V$.

\paragraph{Persistent Local Systems.}  
For every path-connected 
open set $U\subseteq M$, we will construct (see Example \ref{ex:bisheaf} followed by Example
\ref{ex:stack_example}) a local system $\mathcal L(U)$ over $U$ 
called the {\em persistent local system} over $U$ with the following properties:

\begin{enumerate}
\item{Relation to homology of fibers}: For every point $p\in U$, the stalk $\mathcal L(U)_p$ of $\mathcal L(U)$ at $p$ is naturally a subquotient of $\Hfunc_j(f^{-1}p)$, the $j$-th homology of the fiber over~$p$;
see Example \ref{ex:stack_example}.
Recall a subquotient of $\Hfunc_j(f^{-1}p)$ is a quotient $\frac{B}{A}$ where $A \subseteq B \subseteq \Hfunc_j(f^{-1}p)$
are subgroups.

\item{Stability}: For every perturbed $\tilde f$ that is  $\epsilon$ close to $f$,  $\mathcal L(U)|_{U^\epsilon}$ is naturally a subquotient of 
$ \tilde{\mathcal{L}}(U^\epsilon) $ where $U^\epsilon$ is the interior of the subset of $U$ consisting of points that are 
at least a distance $\epsilon$ from the boundary of $U$, $\mathcal L(U)|_{U^\epsilon}$ 
is the restriction of $\mathcal L(U)$ to $U^\epsilon$, and 
$ \tilde{\mathcal{L}}(U^\epsilon)$ is the persistent local system over $U^\epsilon$ constructed from~$\tilde f$;
see the discussion after Corollary \ref{corr:main}.
\end{enumerate}
We define $\mathcal P(U)$, the persistent dimension of $U$, to be the dimension of the stalks of 
$\mathcal L(U)$.  
The two properties of the persistent dimension above follow from the two properties of a persistent local system since the dimension of any vector space $V$ is bounded from below by the dimension any subquotient of $V$. 

\paragraph{Sheaves and Cosheaves.} It is not surprising that sheaf theory is a useful tool to study these questions.  It was introduced by Leray 75 years ago precisely to study the homology of the fibers of a map.  We develop the sheaf theory we need (constructible sheaves and cosheaves) in Sections \ref{sec:sheaves} and \ref{sec:cosheaves}. 
Local systems are equivalent to certain types of sheaves (see Definition \ref{defn:constructible_sheaf}) 
and cosheaves (see Definition \ref{defn:constructible_cosheaf}).

The $j$-th Leray homology cosheaf of a map $f:X\to M$ (see Example \ref{ex:ordinary_cosheaf}) 
is a cosheaf $\cosheaf{F}_j$ under $M$ that contains the information of the $j$-th homology of the 
fibers $\Hfunc_j(f^{-1}p)$, for all points $p\in M$, all woven together into one algebraic object.  
In the cellular setting, it is amenable to computation \cite{shepard}.
The $j$-th Leray relative homology sheaf of~$f$ (see Example \ref{ex:relative_sheaf})
is a similar dual object $\sheaf{F}_j$ over $M$.

\paragraph{The Case $M= \mathbb{R}$ and Persistent Homology.}  If the manifold $M$ is the space of real numbers, then there is a remarkably simple construction of the persistent local systems $\mathcal L(U)$.  Let $\cosheaf{F}$ be the Leray cosheaf of $f$ and $\cosheaf{F}|_U$ be its restriction to $U$.  
Then $\mathcal L(U)$  can be characterized as the largest local system contained in $\cosheaf{F}|_U$ as a direct summand. 
If $\sheaf{F}$ is the Leray sheaf, $\mathcal{L}(U)$ can also be characterized as the largest local system contained in 
$\sheaf{F}|_U$ as a direct summand.

So $\mathcal L(U)$ constructed in this way satisfies the two properties \emph{Relation to homology of the fibers} and \emph{Stability}.  This construction and these properties of it were already known to the 
persistent homology community \cite{levelset_stability, patel}.
Since U is path-connected and simply connected, the stalks of the local system $\mathcal L(U)$ are all identified with a single vector space $V$. 

Most of the persistent homology literature focuses on a special case of our situation.  There is a space $Y$ with a function $h:Y\to \mathbb R$ and we are interested in the homology of the sublevel sets $h^{-1}(-\infty,r]$ as a function of $r$ .  
For every pair $r \leq s$, the image of the homomorphism $\Hfunc_j \big( h^{-1}(-\infty, r] \big)
\to \Hfunc_j \big( h^{-1}(-\infty, s] \big)$ is called the \emph{persistent vector space} associated to the interval $(r,s)$~\cite{Edelsbrunner2002}.
The collection of all dimensions of persistent vector spaces, called the \emph{rank function} of $h$, 
uniquely defines what is called the \emph{barcode} or the \emph{persistence diagram} of $h$ 
\cite{CSEdH, Carlsson2009, Patel2018}.
This special case translates into a case of ours by concocting a function $f:X\to \Rspace$ such that the 
sublevel sets of $h$ are the fibers of $f$.  
Take 
$X= \big\{(y,r)\in Y \times \mathbb R \; \big| \; h(y)\leq r \big\}$ 
and $f(y,r)=r$.   
The persistent vector space of $h$ for an interval $(r,s)$ is the persistent local system
of $f$ over $(r,s)$.
In this way, the persistent local system behaves very much like the well known rank function 
in persistent homology.

There is work on the persistent homology of circle valued functions $f : Y \to \mathbb{S}^1$ \cite{Burghelea2013}.
We believe the persistent local systems of $f$ are closely related to their invariant.

\paragraph{This paper.}This paper was motivated by  our desire to generalize this very beautiful theory of persistent vector spaces to functions with values in any manifold.  
One might ask, why not just do the same thing --  The construction of the persistent local system
$\mathcal L(U)$, for the case $M = \Rspace$ above, 
makes sense for any manifold $M$.  
However, it does not work. 
The result does not satisfy the stability condition.
This is the first indication of many aspects of the problem that are much more complicated for higher dimensional manifolds than for $\mathbb R$. 
In fact, one can show that there can be no construction of persistent local systems $\mathcal L(U)$ 
that depends only on $\sheaf{F}$, gives the ``right'' answer for fibrations, and satisfies stability (see Example \ref{ex:two}); the situation is similar for $\cosheaf{F}$ (see Example \ref{ex:one}). 

Our construction of persistent local systems uses both the cosheaf $\sheaf{F}$ 
and the sheaf $\cosheaf{F}$ plus a map between them
$\Ffunc: \sheaf{F}\to \cosheaf{F}$ (see Example \ref{ex:bisheaf}) constructed from the orientation class of~$M$.  
We call this data a {\em bisheaf}. 
In terms of computability, a bisheaf is not much more complicated than a sheaf or a cosheaf.   
However, since the map $\Ffunc$ mixes objects from different categories, the theory of bisheaves is complicated.  
For example, bisheaves form an interesting category (see Definition \ref{defn:bisheaves}), but unlike sheaves and cosheaves, it is not an abelian category (e.g.\ no zero object).

Given the bisheaf $\Ffunc: \sheaf{F} \to \cosheaf{F}$, the construction of the persistent local system $\mathcal L(U)$ proceeds in four steps.  Here is an outline:
\begin{enumerate}
\item Restrict the bisheaf to $U$, $\Ffunc |_U: \sheaf{F}|_U\to \cosheaf{F}|_U$.
\item Construct a canonical subsheaf $\Epi(\sheaf{F}|_U) \hookrightarrow \cosheaf{F}|_U$;
see Definition \ref{defn:epification}.
\item Construct a canonical quotient cosheaf $\cosheaf{F}|_U \twoheadrightarrow \Mono (\cosheaf{F}|_U)$;
see Definition \ref{defn:monofication}.
\item Then $\mathcal L(U)$ is the image of the composition 
$\Epi(\sheaf{F}|_U) \hookrightarrow \sheaf{F}|_U\to \cosheaf{F}|_U \twoheadrightarrow \Mono (\cosheaf{F}|_U)$;
see Proposition \ref{prop:image_local_system}.
\end{enumerate}

We would have liked the persistent local systems $\mathcal L(U)$ to satisfy a stacky functoriality in $U$
as in \cite{treumann}.  What is true is a rather weaker statement: if $U'$ is a subset of $U$, then  $\mathcal L(U)|_{U'}$ is naturally a subquotient of $\mathcal L(U')$.  The solution we found to this is the {\em isobisheaf stack} (see Definition \ref{defn:isobisheaf_stack}) which has all the functorial properties we need.  We believe that the category of bisheaves, the $\Epi$ and $\Mono$ constructions, and isobisheaf stacks are interesting new tools of sheaf theory.
We hope they will be useful in other contexts.

\paragraph{Acknowledgments.}
The second author thanks Vidit Nanda and Oliver Vipond for carefully reading the first versions of this
paper.
The second author also thanks Justin Curry for helpful comments on Appendices \ref{sec:sheafification}
and \ref{sec:cosheafification}.
Finally, we thank our anonymous reviewers.

\section{Constructible Maps}
\label{sec:maps}
We start by defining the class of spaces and maps we will be working with. The class we consider is chosen to be general enough to include all the maps that generally come up in geometry and applied mathematics, but controlled enough to allow the powerful technology of constructible sheaf theory.  

\begin{defn}{\cite{Mather}}
A \define{Thom-Mather space} is a triple $(X, \Sstrat, \Jcontrol)$ satisfying the 
following nine axioms:
	\begin{enumerate}
	\item $X$ is a Hausdorff, locally compact, and second-countable
	topological space.
	
	\item $\Sstrat$ is a set of path-connected, locally closed subsets of 
	$X$ such that $X$
	is the disjoint union of the elements of $\Sstrat$.
	
	\item[] The elements of $\Sstrat$ are called the \emph{strata} of $X$.
	We call $\Sstrat$ the stratification of the Thom-Mather space.
	
	\item Each stratum of $X$ is a topological manifold (in the induced topology)
	provided with a $C^\infty$ smoothness structure.
	
	\item The set $\Sstrat$ is locally finite. That is, each point $x \in X$ has an open
	neighborhood that intersects finitely many strata.
	
	\item The set $\Sstrat$ satisfies the \emph{condition of the frontier}:
	if $R,S \in \Sstrat$ and $S$ has a non-empty intersection with the
	closure of $R$, then $S$ is a subset of the closure of $R$.
	In this case, we say $S$ is on the \emph{frontier} of $R$.
	
	\item[] The axiom of the frontier makes $\Sstrat$ a poset with $S \leq R$ iff
	$S$ is on the frontier of $R$.
	
	\item $\Jcontrol$ is a triple $\big \{ (T_S), (\pi_S), (\rho_S) \big \}$, where
	for each $S \in \Sstrat$, $T_S$ is an open neighborhood of $S$ in $X$,
	$\pi_S : T_S \to S$ is a continuous retraction onto $S$, and $\rho_S : T_S \to [0,\infty)$
	is a continuous function.
	
	\item[] The open set $T_S \subseteq X$ is called the \emph{tubular neighborhood} 
	of $S$ in $X$,
	$\pi_S$ is called the \emph{local retraction} of $T_S$ onto $S$, and
	$\rho_S$ is called the \emph{tubular function} of $S$.
	We call $\Jcontrol$ the \emph{control data} of the Thom-Mather space.
	
	\item For each stratum $S \in \Sstrat$, $S = \{ x \in T_S \; | \; \rho_S(x) = 0 \}$.
	
	\item[] For two strata $R, S \in \Sstrat$, let $T_{R,S} = T_R \cap S$,
	$\pi_{R,S} = \pi_R | T_{R,S} : T_{R,S} \to R$, and $\rho_{R,S} = \rho_R | T_{R,S}:
	T_{R,S} \to [0,\infty)$.
	It is possible that $T_{R,S}$ is empty, in which case these maps are the
	empty mappings.

	\item For any strata $R,S \in \Sstrat$, the mapping
		$$ (\pi_{R, S} , \rho_{R,S} ) : T_{R, S} \to R \times (0, \infty)$$
		is a smooth submersion.
		
	\item For any strata $Q,R, S \in \Sstrat$, the following diagrams commute:
		\begin{equation*}
		\xymatrix{
		T_{Q,S} \cap T_{R,S} \cap \pi^{-1}_{R,S}(T_{Q,R}) 
		\ar[rd]^>>>>>>>>>{\pi_{Q,S}} \ar[rr]^>>>>>>>>>>>>>>>{\pi_{R,S}} && T_{Q,R} \ar[dl]^{\pi_{Q,R}} &
		T_{R,S} \ar@{^{(}->}[rr]^{\pi_{R,S}} \ar[dr]^{\rho_{Q,S}} && T_R \ar[dl]^{\rho_{Q,R}} \\
		& T_{Q,S} & & & [0, \infty). & 
		}
		\end{equation*}
	
	\end{enumerate}

\end{defn}

Let $(X, \Sstrat, \Jcontrol)$ be a Thom-Mather space.
Choose a stratum $S \in \Sstrat$ and a topological ball $B \subseteq S$
open in $S$.
For a value $r \in (0,\infty)$, let
	$$B_r = \left \{ x \in T_S \; \middle | \; \rho_S(x) < r \text{ and } \pi_S(x) \in B \right\}.$$
We call $B_r$ a \emph{basic open} of $(X, \Sstrat, \Jcontrol)$ \emph{associated}
to the stratum $S$.
Let $\Basic(X, \Sstrat, \Jcontrol)$ be the poset of all basic opens
over all strata $S \in \Sstrat$ and over all $r \in (0, \infty)$ ordered by inclusion.
The union of the open sets in $\Basic(X, \Sstrat, \Jcontrol)$ is $X$.
For any two $U, V \in \Basic(X, \Sstrat, \Jcontrol)$ with $x \in U \cap V$, there is a set
$W \in \Basic(X, \Sstrat, \Jcontrol)$ such that $x \in W$ and $W \subseteq U \cap V$.
This makes $\Basic(X, \Sstrat, \Jcontrol)$ a basis for the topology on $X$.

\begin{defn}
\label{defn:constructible_map}
Let $X$ and $Y$ be Hausdorff, locally compact, and second countable topological spaces.
A continuous map $f : Y \to X$ is \define{$(\Sstrat, \Jcontrol)$-constructible}
if there is a Thom-Mather space $(X, \Sstrat, \Jcontrol)$ such that for every 
pair $V \subseteq U$ in $\Basic(X, \Sstrat, \Jcontrol)$ associated to a common stratum, the inclusions
	\begin{equation*}
	\xymatrix{
	\big( Y, Y - f^{-1}(U) \big) \mono \big( Y, Y - f^{-1}(V) \big)
	&&
	f^{-1}(V) \mono f^{-1}(U)
	}
	\end{equation*}
are homotopy equivalences.
A continuous map $f : Y \to X$ is \define{constructible}
if it is $(\Sstrat, \Jcontrol)$-constructible for some Thom-Mather space
$(X, \Sstrat, \Jcontrol)$.
\end{defn}

\begin{ex}
The following classes of maps are all constructible:
(a) Real algebraic maps, (b) real analytic maps that are ``controlled at infinity,'' 
(c) piecewise linear maps that are ``controlled at infinity,'' 
and (d) an open dense set of  proper smooth maps.

Here ``controlled at infinity'' means that the map $Y\to X$ factorizes in the category of analytic 
(resp.\ PL spaces) as follows: 
$Y\subset Z \to X$ where $Y\subset Z$ is an inclusion of an open set,  $Z-Y$ is analytic (resp.\ 
PL) subspace of $Z$, and  $Z \to X$ is proper. 
Proper maps are automatically controlled at infinity: set $Z=Y$.
Algebraic maps are always similarly controlled at infinity.  

In all four cases, the proof has three steps:  
\begin{enumerate}
\item Construct a Whitney stratified structure on the map $Y \to X$ in which $Y$ is a union of strata, using \cite{Shiota} in cases (a), (b), and (c) and \cite{Gibson} in case (d).
\item Choose the Thom-Mather data on $X$ to be the one obtained from the Whitney stratification of $X$ in \cite{Mather}.
\item Use {\em moving the wall} from \cite[Chapter 4 page 70]{SMT} to show the required homotopy equivalences. 
\end{enumerate}
\end{ex}

\begin{rmk}
We expect almost any map defined by a finite process to be constructible.  Non-constructible examples, like the inclusion of a Cantor set into a manifold, come from infinite or iterative processes.
\end{rmk}

We will not require the smooth structure of a Thom-Mather space until 
Section \ref{sec:dilation}.
For the next few sections, all we require is a topological stratified space.
Recall the open cone $C(X)$ on a topological space $X$ is the quotient space
$\dfrac{X \times [0, \infty)}{X \times \{0\}}$.
Its cone point, denoted $\bullet \in C(X)$, is the point $\dfrac{X \times \{0\}}{X \times \{0\}}$.

\begin{defn}{\cite{GorMac80}}
\label{defn:stratifed_space}
An $n$-\define{dimensional (topological) stratified space} $X$ is an $n$-step filtration
\begin{equation*}
	\emptyset = X_{-1} \subseteq \cdots \subseteq X_n = X
\end{equation*}
of a second countable, locally compact, Hausdorff space 
where for each $d$ and each point $p \in X_d - X_{d-1}$,
there is a compact $(n-d-1)$-dimensional stratified space $L$
and a filtration preserving homeomorphism 
$$h: \Rspace^d \times C(L) \to U$$
such that $U$ is an open neighborhood of $p$ and $h (0, \bullet) = p$.
Here $\Rspace^d$ is interpreted as a filtered space with just one step and~$\bullet$ is the cone point of $C(L)$.
We call $h$ a \emph{local parameterization} of the stratified space.
Each path-connected component of $X_d - X_{d-1}$ is a \define{$d$-stratum}.
It will be convenient to write a stratified space as a tuple 
$(X, \Sstrat)$ where~$\Sstrat$ is its set of strata.
\end{defn}

Let $(X, \Sstrat)$ be an $n$-dimensional stratified space.
The local parameterizations imply that each $d$-stratum is a
topological $d$-manifold and that the condition of the frontier is satisfied.
This makes $\Sstrat$ a poset.
We call an open set $U \subseteq X$ an \emph{$\Sstrat$-basic open} if it is the image
of a local parameterization $h : \Rspace^d \times C(L) \to X$.
An $\Sstrat$-basic open is \emph{associated} to the unique stratum in $\Sstrat$
containing the topological ball $h(\Rspace^d \times \bullet)$.
Let $\Basic(X, \Sstrat) \subseteq \Open(X)$ be the poset of $\Sstrat$-basic opens ordered by inclusion.
The set $\Basic(X, \Sstrat)$ is a basis for the topology on $X$.
Note that every finite dimensional Thom-Mather space $(X, \Sstrat, \Jcontrol)$ is a stratified space $(X, \Sstrat)$
and $\Basic(X, \Sstrat, \Jcontrol) \subseteq \Basic(X, \Sstrat)$.
However, not every open set in $\Basic(X, \Sstrat)$ belongs to $\Basic(X, \Sstrat, \Jcontrol)$.

For the purpose of proving stability (see Theorem \ref{thm:main}), it will be convenient to work with a 
triangulation of a Thom-Mather space as opposed to working directly with a Thom-Mather space.
Recall a pair $(K, K_0)$ of simplicial complexes is a simplicial complex $K$ and a subcomplex $K_0 \subseteq K$.
The geometric realization $|K - K_0|$ of the pair is the geometric realization $|K|$ take-away the subspace
$|K_0| \subseteq |K|$.

\begin{defn}
A stratified space $(X, \Kstrat)$ is a \define{triangulation}
if there is a simplicial pair $(K, K_0)$ and a homeomorphism $\phi : | K - K_0 | \to X$
such that each stratum of $\Kstrat$ is the image of a simplex in $K - K_0$.
A stratified space $(X, \Sstrat)$ is \define{triangulable} if there is a triangulation
$(X, \Kstrat)$ such that for each stratum $\sigma \in \Kstrat$ there is a stratum
$S \in \Sstrat$ where $\sigma \subseteq S$.
\end{defn}

We use $\sigma$ and $\tau$ to denote strata of a triangulation $(X, \Kstrat)$.
The \emph{open star} of a stratum $\sigma \in \Kstrat$ is the subposet
$\st\; \sigma := \big \{  \tau \in \Kstrat \; | \; \sigma \leq \tau \big \} \subseteq \Kstrat.$
Note that every open star $\st\; \sigma$ is a $\Kstrat$-basic open associated
to the stratum $\sigma$.


\begin{prop}[\cite{10.2307/2042563}]
\label{prop:triangulation}
Every Thom-Mather space $(X, \Sstrat, \Jcontrol)$ is triangulable.
\end{prop}

Throughout this paper, $M$ will denote a topological $m$-manifold without boundary.
A topological manifold is a locally Euclidean, second-countable, and Hausdorff space.

\section{Sheaves}
\label{sec:sheaves}
In this section, we develop the theory of constructible sheaves.  We introduce the notions of an episheaf and epification which we will use to study the fibers of a 
constructible~map.  On a technical level, the main new device is the use of basic open sets. 

For a topological space $X$, 
let $\Open(X)$ be its poset of open sets ordered
by inclusion~$V \subseteq U$.
An \emph{open cover} of an open set $U \subseteq X$ is 
a subposet $\Ucat \subseteq \Open(X)$ of open sets whose union is $U$ and 
for every $U_i, U_j \in \Ucat$, $U_i \cap U_j$ is a union of elements in $\Ucat$.
Let $\Ab$ be the category of abelian groups.

\begin{defn}
\label{defn:sheaf}
A \define{sheaf} (of abelian groups) \emph{over} $X$ is a 
contravariant functor
$$\sheaf{F} : \Open(X) \to \Ab$$
satisfying the following property.
For every open set $U \subseteq X$ and for every open cover $\cov{U}$ of $U$, 
the universal map $\sheaf{F}(U) \to \mylim \sheaf{F} |_{\cov{U}}$
is an isomorphism.
A \define{sheaf map} is a natural transformation of functors 
$\ubar{\alpha} : \sheaf{F} \to \sheaf{G}$.
\end{defn}


\begin{defn}
\label{defn:constructible_sheaf}
Let $(X, \Sstrat)$ be a stratified space.
A sheaf $\sheaf{F}$ over $X$ is \define{$\Sstrat$-constructible} 
if for every pair of $\Sstrat$-basic opens $V \subseteq U$
associated to a common stratum,
the map
	$$\sheaf{F}(V \subseteq U) : \sheaf{F}(U) \to \sheaf{F}(V)$$
is an isomorphism.
If $\sheaf{F}(V \subseteq U)$ is an isomorphism for \emph{every} pair of $\Sstrat$-basic opens 
$V \subseteq U$, then $\sheaf{F}$ is a \define{local system}.
A sheaf $\sheaf{F}$ over $X$ is \define{constructible} if there is a stratified space
$(X, \Sstrat)$ for which $\sheaf{F}$ is $\Sstrat$-constructible.
Let $\Sheaf(X, \Sstrat)$ be the category of $\Sstrat$-constructible 
sheaves over $X$ and sheaf maps.
Let $\Sheaf(X)$ be the category of constructible sheaves
over $X$ and sheaf maps.
\end{defn}

When defining an $\Sstrat$-constructible sheaf over $X$, it is enough to specify
a well behaved contravariant functor on a subposet of $\Open(X)$.
Let $\Acat \subseteq \Basic(X, \Sstrat)$ be any subposet that is a basis for the topology on $X$.
For example, if $(X, \Sstrat, \Jcontrol)$ is a Thom-Mather space, then we may let
$\Acat$ be $\Basic(X, \Sstrat, \Jcontrol)$.
Let $\Ffunc : \Acat \to \Ab$ be a contravariant functor
such that for every pair $J \subseteq I$ associated to a common stratum, the map
$\Ffunc(J \subseteq I)$ is an isomorphism.
Then $\Ffunc$ uniquely generates (up to an isomorphism) 
an $\Sstrat$-constructible sheaf $\sheaf{F}$ as follows.
For an arbitrary open set $U \subseteq X$, let $\Acat(U) \subseteq \Acat$ 
be the subposet consisting of all open sets contained in $U$.
Let $\sheaf{F}(U) := \lim \Ffunc |_{\Acat(U)}$.
Note that if $U \in \Acat$, then $\sheaf{F}(U)$ is canonically isomorphic to $\Ffunc(U)$.
For an arbitrary pair of open sets $V \subseteq U \subseteq X$, let $\sheaf{F}(V \subseteq U)$ be the
universal morphism between the two limits.
See Appendix \ref{sec:sheafification} for a check that $\sheaf{F}$ is indeed an $\Sstrat$-constructible sheaf.

Given a sheaf map $\ubar \alpha : \sheaf{F} \to \sheaf{G}$ between two $\Sstrat$-constructible
sheaves, its image $\image \ubar \alpha$ is an $\Sstrat$-constructible sheaf equipped with a canonical
inclusion $\image \ubar \alpha \mono \sheaf{G}$ as follows.
Let $\Hfunc : \Basic(X, \Sstrat) \to \Ab$ be the contravariant functor that assigns to 
each $\Sstrat$-basic open $U$ the group $\Hfunc(U) := \image \ubar \alpha(U)$.
For $V \subseteq U$, the map $\Hfunc(V \subseteq U)$ is the map $\sheaf{G}(V \subseteq U)$
restricted to $\image \ubar \alpha(U)$.
If both $V \subseteq U$ are associated to a common stratum, then $\Hfunc(V \subseteq U)$
is an isomorphism. 
Extend $\Hfunc$ to a sheaf $\sheaf{H}$ over $X$ using the procedure in the previous paragraph.
For any open set $U \subseteq X$, the universal morphism $\sheaf{H}(U) \to \sheaf{G}(U)$
is injective.
The coimage, kernel, and cokernel of $\ubar \alpha$ are defined similarly.



\begin{ex}
\label{ex:relative_sheaf}
Let $f : Y \to X$ be an $(\Sstrat, \Jcontrol)$-constructible map.
Define $\sheaf{F}_\ast$ as the $\Sstrat$-constructible sheaf generated
by assigning to each  $U \in \Basic(X, \Sstrat, \Jcontrol)$
the relative singular homology group
$$\sheaf{F}_\ast ( U ) := \Hgroup_{\ast} \big( Y, Y - {f}^{-1} ( U ); \Zspace \big).$$
For two $(\Sstrat, \Jcontrol)$-basic opens $V \subseteq U$ associated to a common
stratum, the map
	$$\sheaf{F}_\ast(V \subseteq U) : \sheaf{F}_\ast (U) \to \sheaf{F}_\ast(V)$$
is, by definition of an $(\Sstrat, \Jcontrol)$-constructible map, an isomorphism.
Thus $\sheaf{F}_\ast$ is an $\Sstrat$-constructible sheaf.
\end{ex}


\begin{defn}
An $\Sstrat$-constructible sheaf $\sheaf{F}$ over $X$ is an
\define{episheaf} if for every pair of $\Sstrat$-basic opens 
$V \subseteq U$,
the map $\sheaf{F}(V \subseteq U) : \sheaf{F}(U) \to \sheaf{F}(V)$
is surjective.
\end{defn}

\begin{prop}
\label{prop:epi_one}
Consider a sheaf map $\ubar \alpha : \sheaf{E} \to \sheaf{F}$ in $\Sheaf(X, \Sstrat)$.
If $\sheaf{E}$ is an episheaf, then $\image \ubar \alpha$ is an $\Sstrat$-constructible episheaf.
\end{prop}
\begin{proof}
By construction of $\image \ubar \alpha$, we need only look at the following
commutative diagram for any pair of $\Sstrat$-basic opens $V \subseteq U$:
	\begin{equation*}
	\xymatrix{
	\sheaf{E}(U) \ar[d]_{\ubar \alpha(U)} \ar@{->>}[rr]^{\sheaf{E}(V \subseteq U)} && 
	\sheaf{E}(V) \ar[d]^{\ubar \alpha(V)} \\
	\sheaf{F}(U) \ar[rr]^{\sheaf{F}(V \subseteq U)} && \sheaf{F}(V).
	}
	\end{equation*}
The restriction of $\sheaf{F}(V \subseteq U)$ to the image of $\ubar \alpha(U)$
is a surjection onto the image of $\ubar \alpha (V)$.
Thus $\image \ubar \alpha$ is an episheaf.
\end{proof}

Let $\sheaf{F}$ be an $\Sstrat$-constructible sheaf over $X$.
A \emph{sub-episheaf} of $\sheaf{F}$ is an inclusion $\sheaf{E} \mono \sheaf{F}$
of an $\Sstrat$-constructible episheaf $\sheaf{E}$.
The zero sheaf $\sheaf{0} \mono \sheaf{F}$ is the smallest sub-episheaf of $\sheaf{F}$.
For any two sub-episheaves 
$\sheaf{E}_1, \sheaf{E}_2 \mono \sheaf{F}$, their internal sum 
$\sheaf{E}_1 \uplus \sheaf{E}_2$, which assigns to each open set $U$ the smallest subgroup of $\sheaf{F}(U)$
containing both $\sheaf{E}_1(U)$ and $\sheaf{E}_2(U)$, is also a sub-episheaf.
Let $\Pcat$ be the poset of sub-episheaves of $\sheaf{F}$ ordered by inclusion.
For any chain
	\begin{equation*}
	\xymatrix{
	\sheaf{E}_1 \ar@{^{(}->}[r] \ar@{^{(}->}[rd] & \sheaf{E}_2 \ar@{^{(}->}[r] \ar@{^{(}->}[d] & 
	\sheaf{E}_3 \ar@{^{(}->}[r] \ar@{^{(}->}[ld] & \cdots \\
	& \sheaf{F} && 
	}
	\end{equation*}
in $\Pcat$, the sub-episheaf $\biguplus \sheaf{E}_i$ contains them all.
By Zorn's Lemma, $\Pcat$ has a maximal element and therefore 
$\sheaf{F}$ has a maximal sub-episheaf.
Consider a sheaf map $\ubar{\alpha} : \sheaf{F} \to \sheaf{G}$ in~$\Sheaf(X, \Sstrat)$.
Suppose $\sheaf{D} \mono \sheaf{F}$ and $\sheaf{E} \mono \sheaf{G}$
are maximal sub-episheaves.
By Proposition~\ref{prop:epi_one}, 
the image of the composition
	\begin{equation*}
	\xymatrix{
	\sheaf{D} \ar@{^{(}->}[r] & \sheaf{F} \ar[r]^{\ubar{\alpha}} & \sheaf{G}
	}
	\end{equation*}
is a sub-episheaf of $\sheaf{G}$.
My maximality of $\sheaf{E}$, this image is contained in $\sheaf{E}$ thus inducing
a map $\sheaf{D} \to \sheaf{E}$ that makes the following diagram commute:
	\begin{equation*}
	\xymatrix{
	\sheaf{D} \ar@{^{(}->}[d] \ar@{-->}[rr] && \sheaf{E} \ar@{^{(}->}[d] \\
	\sheaf{F} \ar[rr]^{\ubar{\alpha}} && \sheaf{G}.
	}
	\end{equation*}
Thus the assignment to each $\Sstrat$-constructible sheaf 
its maximal sub-episheaf is functorial.

\begin{defn}
\label{defn:epification}
The \define{epification} of $\Sstrat$-constructible sheaves over $X$ is the functor
	$$\Epi : \Sheaf(X, \Sstrat) \to \Sheaf( X, \Sstrat)$$
that sends each sheaf to its maximal sub-episheaf. 
Let $\ubar{\eta} : \Epi \Rightarrow \id_{\Sheaf(X, \Sstrat)}$ be the inclusion
natural transformation.
\end{defn}

\section{Cosheaves}
\label{sec:cosheaves}
Cosheaves are ``dual'' to their better known cousins, sheaves.  In this section, whose parallel structure to the last one reflects that ``duality," we develop the theory of constructible cosheaves.  We introduce the notions of a monocosheaf and monofication.

\begin{defn}
\label{defn:cosheaf}
A \define{cosheaf} (of abelian groups) \emph{under} $X$ is a covariant functor
$$\cosheaf{F} : \Open(X) \to \Ab$$
satisfying the following property.
For every open set $U \subseteq X$ and for every open cover $\cov{U}$ of $U$,
the universal map 
$\colim\; \cosheaf{F} |_{\cov{U}} \to \cosheaf{F}(U)$
is an isomorphism.
A \define{cosheaf map} is a natural transformation of functors $\lbar{\alpha} : \cosheaf{F} \to \cosheaf{G}$.
\end{defn}

\begin{defn}
\label{defn:constructible_cosheaf}
Let $(X, \Sstrat)$ be a stratified space.
A cosheaf $\cosheaf{F}$ under $X$ is \define{$\Sstrat$-constructible} 
if for every pair of open sets $V \subseteq U$ in $\Basic(X, \Sstrat)$ associated
to a common stratum, the map 
$$\cosheaf{F}(V \subseteq U) : \cosheaf{F}(V) \to \cosheaf{F}(U)$$
is an isomorphism.
A cosheaf $\cosheaf{F}$ under $X$ is \define{constructible} if it is
$\Sstrat$-constructible for some stratified space $(X, \Sstrat)$.
If $\cosheaf{F}(V \subseteq U)$ is an isomorphism for \emph{every} pair of $\Sstrat$-basic opens 
$V \subseteq U$, then $\cosheaf{F}$ is a \define{colocal system}.
Let $\Cosheaf(X, \Sstrat)$ be the category of $\Sstrat$-constructible
cosheaves under $X$ and cosheaf maps.
Let $\Cosheaf(X)$ be the category of constructible cosheaves under $X$ and cosheaf maps.
\end{defn}

When defining an $\Sstrat$-constructible cosheaf under $X$, it is enough to specify
a well behaved covariant functor on a subposet of $\Open(X)$.
Let $\Acat \subseteq \Basic(X, \Sstrat)$ be any subposet that is a basis for the
topology on $X$.
For example, if $(X, \Sstrat, \Jcontrol)$ is a Thom-Mather space, then we may let
$\Acat$ be $\Basic(X, \Sstrat, \Jcontrol)$.
Let $\Ffunc : \Acat \to \Ab$ be a covariant functor 
such that for every pair $J \subseteq I$ associated to a common stratum, the map
$\Ffunc(J \subseteq I)$ is an isomorphism.
Then $\Ffunc$ uniquely generates (up to an isomorphism) an $\Sstrat$-constructible cosheaf 
$\cosheaf{F}$ as follows.
For an arbitrary open set $U \subseteq X$, let $\Acat(U) \subseteq \Acat$ 
be the subposet consisting of all open sets in $U$.
Let $\cosheaf{F}(U) := \colim \Ffunc |_{\Acat(U)}$.
Note that is $U \in \Acat$, then $\cosheaf{F}(U)$ is canonically isomorphic to $\Ffunc(U)$.
For an arbitrary pair of open sets $V \subseteq U \subseteq X$, 
let $\cosheaf{F}(V \subseteq U)$ be the
universal morphism between the two colimits.
See Appendix \ref{sec:cosheafification} for a check that $\cosheaf{F}$ is indeed an $\Sstrat$-constructible cosheaf.

Given a cosheaf map $\ubar \alpha : \cosheaf{F} \to \cosheaf{G}$ between two $\Sstrat$-constructible
cosheaves, its image $\image \lbar \alpha$ is an $\Sstrat$-constructible cosheaf equipped with a canonical
inclusion $\image \lbar \alpha \mono \sheaf{G}$ as follows.
Let $\Hfunc : \Basic(X, \Sstrat) \to \Ab$ be the covariant functor that assigns to 
each $\Sstrat$-basic open $U$ the group $\Hfunc(U) := \image \lbar \alpha(U)$.
For $V \subseteq U$, the map $\Hfunc(V \subseteq U)$ is the map $\cosheaf{G}(V \subseteq U)$
restricted to $\image \lbar \alpha(V)$.
If both $V \subseteq U$ are associated to a common stratum, then $\Hfunc(V \subseteq U)$
is an isomorphism. 
Extend $\Hfunc$ to a cosheaf $\cosheaf{H}$ under $X$ using the procedure in the previous paragraph.
For any open set $U \subseteq X$, the universal morphism $\cosheaf{H}(U) \to \cosheaf{G}(U)$
is injective.
The coimage, kernel, and cokernel of $\lbar \alpha$ are defined similarly.



\begin{ex}
\label{ex:relative_cosheaf}
Let $f : Y \to X$ be a $(\Sstrat, \Jcontrol)$-constructible map.
Define $\cosheaf{F}^\ast$ as the $\Sstrat$-constructible cosheaf generated by assigning 
to each $U \in \Basic(X, \Sstrat, \Jcontrol)$ the 
singular relative cohomology group
$$\cosheaf{F}^\ast ( U ) := \Hgroup^\ast \big( Y, Y - f^{-1} (U) ; \Zspace \big).$$
For two $(\Sstrat, \Jcontrol)$-basic opens $V \subseteq U$ associated to a common stratum, the map
	$$\cosheaf{F}^\ast(V \subseteq U) : \cosheaf{F}^\ast (V) \to \cosheaf{F}^\ast(U)$$
is, by definition of an $(\Sstrat, \Jcontrol)$-constructible map, an isomorphism.
Thus $\cosheaf{F}^\ast$ is an $\Sstrat$-constructible cosheaf.
\end{ex}

\begin{ex}
\label{ex:ordinary_cosheaf}
Let $f : Y \to X$ be a $(\Sstrat, \Jcontrol)$-constructible map.
Define $\cosheaf{F}_\ast$ as the $\Sstrat$-constructible cosheaf generated by assigning 
to each $U \in \Basic(X, \Sstrat, \Jcontrol)$ the 
singular homology group
$$\cosheaf{F}_\ast ( U ) := \Hgroup_\ast \big( f^{-1}(U) ; \Zspace \big).$$
For two $(\Sstrat, \Jcontrol)$-basic opens $V \subseteq U$ associated to a common stratum,
the map
$$\cosheaf{F}_\ast(V \subseteq U) : \cosheaf{F}_\ast (V ) \to \cosheaf{F}_\ast( U )$$
is, by definition of an $(\Sstrat, \Jcontrol)$-constructible map,
an isomorphism.
Thus $\cosheaf{F}_\ast$ is a $\Sstrat$-constructible cosheaf.
\end{ex}

\begin{ex}
\label{ex:orientation_cosheaf}
Let $(M, \Sstrat)$ be a stratified space where $M$ is an $m$-manifold 
without boundary and $\Sstrat$ consists of a single stratum namely $M$.
Note that an open set is an $\Sstrat$-basic open iff it is an open topological $m$-ball.
The \emph{local orientation cosheaf} under $M$ is the $\Sstrat$-constructible
cosheaf $\cosheaf{O}$ 
generated by assigning to each open topological $m$-ball $U \subseteq M$
the top dimensional singular relative cohomology group 
$$\cosheaf{O}( U ) := \Hfunc^m \big(M, M - U ; \Zspace \big) \cong \Zspace.$$
For two $m$-balls $V \subseteq U$,
the map
$$\cosheaf{O}(V \subseteq U) : \cosheaf{O} (V) \to \cosheaf{O} (U)$$
is an isomorphism.
Thus $\cosheaf{O}$ is an $\Sstrat$-constructible cosheaf.
Moreover, $\cosheaf{O}$ is a colocal system.
The manifold $M$ is \emph{orientable} if $\cosheaf{O}(M) \cong \Zspace$.
If $M$ is orientable, then an \emph{orientation} of $M$ is the choice of a 
generator of $\cosheaf{O}(M)$.
The poset of all $m$-balls $\Basic(M, \Sstrat)$ is a covering of $M$.
By the cosheaf axiom, the universal map $\colim \cosheaf{O} |_{\Basic(M, \Sstrat)} \to \cosheaf{O}(M)$
is an isomorphism.
If $M$ is orientable, then the map $\cosheaf{O}(U \subseteq M)$ is an isomorphism
for all $m$-balls $U$.
\end{ex}

\begin{defn}
An $\Sstrat$-constructible cosheaf $\cosheaf{M}$ under $X$ is a \define{monocosheaf} if for every pair of 
$\Sstrat$-basic opens $V \subseteq U$, 
the map $\cosheaf{M}(V \subseteq U): \cosheaf{M}(V) \to \cosheaf{M}(U)$
is injective.
\end{defn}

\begin{prop}
\label{prop:mono_one}
Consider a cosheaf map $\lbar \alpha : \cosheaf{F} \to \cosheaf{M}$ in $\Cosheaf(X, \Sstrat)$.
If $\cosheaf{M}$ is a monocosheaf, then the image of $\lbar \alpha$ is an $\Sstrat$-constructible
monocosheaf.
\end{prop}
\begin{proof}
By construction of $\image \ubar \alpha$, we need only look at the following
commutative diagram for any pair of $\Sstrat$-basic opens $V \subseteq U$:
	\begin{equation*}
	\xymatrix{
	\cosheaf{F}(U) \ar[d]_{\lbar \alpha (U)}
	&& \cosheaf{F}(V) \ar[d]^{\lbar \alpha(V)} \ar[ll]_{\cosheaf{F}(V \subseteq U)} \\
	\cosheaf{M}(U) && \cosheaf{M}(V) \ar@{^{(}->}[ll]_{\cosheaf{M}(V \subseteq U)}
	}
	\end{equation*}
The restriction of $\cosheaf{M}(V \subseteq U)$ to the image of $\lbar \alpha(V)$
is an injection into the image of $\lbar \alpha(U)$.
Thus $\image \lbar \alpha$ is a monocosheaf.
\end{proof}

Let $\cosheaf{F}$ be an $\Sstrat$-constructible cosheaf under $X$.
A \emph{quotient-monocosheaf} of $\cosheaf{F}$ is a 
surjection $\cosheaf{F} \epi \cosheaf{M}$ to an $\Sstrat$-constructible monocosheaf $\cosheaf{M}$.
The zero cosheaf $\cosheaf{F} \to \cosheaf{0}$ is the largest quotient-monocosheaf of $\cosheaf{F}$
because its kernel is all of $\cosheaf{F}$.
For any two quotient-monocosheaves $\cosheaf{F} \epi \cosheaf{M}_1$ and 
$\cosheaf{F} \epi \cosheaf{M}_2$, let 
$\cosheaf{K}_1, \cosheaf{K}_2 \subseteq \cosheaf{F}$ be their kernels.
Then 
$\cosheaf{F} \epi \sfrac{\cosheaf{F}}{\cosheaf{K}_1 \cap \cosheaf{K}_2}$,
which assigns to each open set $U$ the quotient $\sfrac{\cosheaf{F}(U)}{\cosheaf{K}_1(U) \cap \cosheaf{K}_2(U)}$,
is a quotient-monocosheaf of $\cosheaf{F}$.
Let $P$ be the poset of kernels of quotient-monocosheaves of $\cosheaf{F}$
ordered by containment.
For any chain of quotient-monocosheaves
	\begin{equation*}
	\xymatrix{
	&& \cosheaf{F} \ar@{->>}[ld] \ar@{->>}[d] \ar@{->>}[rd] && \\
	\cdots \ar@{->>}[r] & \cosheaf{M}_3 \ar@{->>}[r] & \cosheaf{M}_2 \ar@{->>}[r] & \cosheaf{M}_1,
	}
	\end{equation*}
the corresponding chain of kernels in $P$ has, by taking intersections, a minimal element in~$P$.
By Zorn's Lemma, $P$ has a minimal element and therefore $\cosheaf{F}$ has a minimal quotient-monocosheaf.
Consider a cosheaf map $\lbar{\alpha} : \cosheaf{F} \to \cosheaf{G}$ in $\Cosheaf(X, \Sstrat)$ and suppose
$\cosheaf{F} \epi \cosheaf{M}$ and $\cosheaf{G} \epi \cosheaf{N}$
are minimal quotient-monocosheaves.
By Proposition \ref{prop:mono_one}, 
the image of the composition
	\begin{equation*}
	\xymatrix{
	\cosheaf{F} \ar[r]^{\lbar \alpha} & \cosheaf{G} \ar@{->>}[r] & \cosheaf{N}
	}
	\end{equation*}
is a quotient-monocosheaf of $\cosheaf{F}$.
By minimality of $\cosheaf{M}$, the kernel of $\cosheaf{F} \epi \cosheaf{M}$
is contained in the kernel of the above composition
inducing a map
$\cosheaf{M} \to \cosheaf{N}$ that makes the following diagram commute:
	\begin{equation*}
	\xymatrix{
	\cosheaf{F} \ar@{->>}[d] \ar[rr]^{\lbar{\alpha}} && \cosheaf{G} \ar@{->>}[d] \\
	\cosheaf{M} \ar@{-->}[rr] && \cosheaf{N}.
	}
	\end{equation*}
Thus the assignment to each $\Sstrat$-constructible cosheaf 
its minimal quotient-monocosheaf is functorial.

\begin{defn}
\label{defn:monofication}
The \define{monofication} of $\Sstrat$-constructible cosheaves under $X$ is the functor
	$$\Mono : \Cosheaf(X, \Sstrat) \to \Cosheaf(X, \Sstrat)$$
that sends each cosheaf to its minimal quotient-monocosheaf.
Let $\lbar{\eta} : \id_{\Cosheaf(X)} \Rightarrow \Mono$ be the quotient
natural transformation.
\end{defn}

\section{Bisheaves}
\label{sec:bisheaves}
We now have both a sheaf theoretic and a cosheaf theoretic approach to studying the fibers 
of a constructible map.  As mentioned in Section \ref{sec:introduction}, neither of these alone is enough to produce the stability results we want.  
We now combine the two approaches with the ideas of a bisheaf and an isobisheaf.

\begin{defn}
Let $\COpen(X) \subseteq \Open(X)$ be the subposet consisting of path-connected
open sets.
A \define{bisheaf} \emph{around} $X$ is a triple
	$\bisheaf{F} := \big( \sheaf{F}, \cosheaf{F}, \Ffunc \big)$
where $\sheaf{F}$ is a sheaf over $X$, $\cosheaf{F}$ is a cosheaf under $X$,
and $\Ffunc := \big \{ \Ffunc(U): \sheaf{F}(U) \to \cosheaf{F}(U) \big \}_{U \in \COpen(X)}$ is a set of maps satisfying the following property.
For for each pair of open sets $V \subseteq U$ in $\COpen(X)$, the following diagram commutes:
	\begin{equation*}
	\xymatrix{
	\sheaf{F}(U) \ar[d]_{\Ffunc(U)} \ar[rr]^{\sheaf{F}(V \subseteq U )} && \sheaf{F}(V) \ar[d]^{\Ffunc(V)} \\
	\cosheaf{F}(U) && \cosheaf{F}(V) \ar[ll]^{\cosheaf{F}(V \subseteq U)}.
	}
	\end{equation*}
A \define{bisheaf map} $\ulbar{\alpha} : \bisheaf{F} \to \bisheaf{G}$ is a pair of maps
$\big( \ubar{\alpha}, \lbar{\alpha} \big)$ where
$\ubar{\alpha}: \sheaf{F} \to \sheaf{G}$ is a sheaf map and $\lbar{\alpha} : \cosheaf{G} \to \cosheaf{F}$
is a cosheaf map satisfying the following property.
For every path-connected open set $U \subseteq X$,
the following diagram commutes:
	\begin{equation*}
	\xymatrix{
	\sheaf{F}(U) \ar[d]_{\Ffunc(U)} \ar[rr]^{\ubar{\alpha}(U)}&& \sheaf{G}(U) \ar[d]^{\Gfunc(U)} \\
	\cosheaf{F}(U)  && \cosheaf{G}(U) \ar[ll]^{\lbar{\alpha}(U)}.
	}
	\end{equation*}
\end{defn}

\begin{defn}
\label{defn:bisheaves}
A bisheaf $\bisheaf{F} = \big( \sheaf{F}, \cosheaf{F}, \Ffunc \big)$ around $X$ 
is \define{$\Sstrat$-constructible} 
if both $\sheaf{F}$ and $\cosheaf{F}$ are $\Sstrat$-constructible.
A bisheaf is \define{constructible} if it is $\Sstrat$-constructible for some stratification
$(X, \Sstrat)$.
Let $\Bisheaf(X, \Sstrat)$ be the category of $\Sstrat$-constructible bisheaves
around $X$ and bisheaf maps.
Let $\Bisheaf(X)$ be the category of constructible bisheaves around
$X$ and bisheaf maps.
\end{defn}

When defining an $\Sstrat$-constructible bisheaf 
$\bisheaf{F} = ( \sheaf{F}, \cosheaf{F}, \Ffunc )$ around $X$, it is enough to
specify the sheaf, cosheaf, and maps between them on a subposet of $\Open(X)$.
Let $\Acat \subseteq \Basic(X, \Sstrat)$ be any subposet that is a basis for the topology on $X$.
For example, if $(X, \Sstrat, \Jcontrol)$ is a Thom-Mather space, then we may let
$\Acat$ be $\Basic(X, \Sstrat, \Jcontrol)$.
Let $\Gfunc : \Acat \to \Ab$ be a contravariant functor
such that for each pair $J \subseteq I$ associated to a common stratum, the map
$\Gfunc(J \subseteq I)$ is an isomorphism.
Let $\Hfunc : \Acat \to \Ab$ be a covariant functor
such that for each pair $J \subseteq I$ associated to a common stratum, the map
$\Hfunc(J \subseteq I)$ is an isomorphism.
Let $\big \{ \Qfunc(J) \big \}_{J \in \Acat}$ be a set of  maps
$\Qfunc(J) : \Gfunc(J) \to \Hfunc(J)$ such that for every $J \subseteq I$,
the following diagram commutes:
	\begin{equation*}
	\xymatrix{
	\Gfunc(I) \ar[d]^{\Qfunc(I)} \ar[rr]^{\Gfunc(J \subseteq I)} && \Gfunc(J) \ar[d]^{\Qfunc(J)} \\
	\Hfunc(I) && \Hfunc(J) \ar[ll]^{\Hfunc(J \subseteq I)}.
	}
	\end{equation*}
Given this data, we define $\bisheaf{F}$ as follows.
For an arbitrary open set $U \subseteq X$, let $\Acat(U) \subseteq \Acat$
be the subposet consisting of all open sets contained in $U$.
The $\Sstrat$-constructible sheaf $\sheaf{F}$ is generated by setting $\sheaf{F}(U) := \lim \Gfunc |_{\Acat}$
as we did in Section \ref{sec:sheaves}.
The $\Sstrat$-constructible cosheaf $\cosheaf{F}$ is generated by setting
$\cosheaf{F}(U) := \colim \Hfunc |_{\Acat}$ as we did in Section \ref{sec:cosheaves}.
Suppose $U$ is path-connected.
For every pair of open sets $J \subseteq I$ in $\Acat(U)$, the following diagram,
where $\pi_J$ is the canonical map from the limit and $\iota_J$ is the canonical map to the
colimit, commutes:
	\begin{equation*}
	\xymatrix{
	& \sheaf{F}(U) := \lim \Gfunc |_{\Acat(U)} \ar[ld]^{\pi_I} \ar[rd]^{\pi_J} & \\
	\Gfunc(I) \ar[d]^{\Qfunc(I)} \ar[rr]^{\Gfunc(J \subseteq I)} && \Gfunc(J) \ar[d]^{\Qfunc(J)} \\
	\Hfunc(I) \ar[dr]^{\iota_I} && \Hfunc(J) \ar[ld]^{\iota_J} \ar[ll]^{\Hfunc(J \subseteq I)} \\
	& \cosheaf{F}(U) := \colim \Hfunc |_{\Acat(U)}. & 
	}
	\end{equation*}
Define $\Ffunc(U) : \sheaf{F}(U) \to \cosheaf{F}(U)$ as any composition, for example $\iota_J \circ \Qfunc(J) \circ \pi_J$,
from the limit to the colimit.
Note that if $U$ is the disjoint union of infinitely many path-connected components, then there is no canonical 
map from the limit to the colimit.
	
\begin{ex}
\label{ex:bisheaf}
Let $f : Y \to M$ be a $(\Sstrat, \Jcontrol)$-constructible map to an oriented $m$-manifold $M$.
Recall the relative homology sheaf $\sheaf{F}_{\ast+m}$ of $f$ and
the ordinary homology cosheaf $\cosheaf{F}_\ast$ of $f$; see
Examples \ref{ex:relative_sheaf} and \ref{ex:ordinary_cosheaf} respectively.
Then there is a constructible bisheaf 
$$\bisheaf{F}_\ast := \Big( \sheaf{F}_{\ast + m} , \cosheaf{F}_\ast, \big \{ \Ffunc_\ast(U) \big \} \Big)$$
around $M$ where, for each $(\Sstrat, \Jcontrol)$-basic open $U$,
$\Ffunc(U)$ is a cap product constructed as follows.

Recall the local orientation cosheaf $\cosheaf{O}$ of $M$; see Example \ref{ex:orientation_cosheaf}.
Fix an orientation $o \in \cosheaf{O}(M)$.
Let $U \subseteq M$ be an $(\Sstrat, \Jcontrol)$-basic open and suppose $U$ is associated to a
stratum $S \in \Sstrat$.
Choose an $(\Sstrat, \Jcontrol)$-basic open $U' \subsetneq U$ that is also associated to $S$.
Then the the inclusion
$$\big( f^{-1}(U), f^{-1}(U) - f^{-1}(U') \big) \mono 
\big( Y, Y - f^{-1}(U') \big)$$
induces, by excision, an isomorphism on their relative singular (co)homology groups.
The inclusion
$$\big( Y, Y - f^{-1}(U) \big) \mono \big( Y, Y - f^{-1}(U') \big)$$
induces, by definition of a constructible map, an isomorphism on their singular relative (co)homology groups.
Thus the singular cap product
	\begin{equation*}
	\xymatrix{
	\Hfunc_{\ast +m}\big( f^{-1}(U), f^{-1}(U) - f^{-1}(U') \big)  \otimes  
	\Hfunc^m \big( f^{-1}(U), f^{-1}(U) - f^{-1}(U') \big)  \ar[r]^-{\frown}  
	& \Hfunc_\ast \big( f^{-1}(U) \big)
	}
	\end{equation*}
gives rise to a map
	\begin{equation*}
	\xymatrix{
	\sheaf{F}_{\ast +m}\big( U \big) \otimes  
	\cosheaf{F}^m \big( U \big)  \ar[r]^-{\frown}  
	& \cosheaf{F}_\ast \big( U \big)
	}
	\end{equation*}
where $\cosheaf{F}^m$ is the cosheaf of relative cohomology groups; see Example \ref{ex:relative_cosheaf}.
For any pair of $(\Sstrat, \Jcontrol)$-basic opens $V \subseteq U$,
we have the following diagram where the vertical 
maps are induced by inclusion:
	\begin{equation*}
	\label{eq:naturality}
	\xymatrix{
	\sheaf{F}_{\ast + m}(U)  \ar@<-5ex>[d]_{i} \otimes  \cosheaf{F}^{m}(U) 
	\ar[r]^<<<<<{\frown} & \cosheaf{F}_{\ast}(U) \\
	\sheaf{F}_{\ast + m}(V)   \otimes \cosheaf{F}^{m}(V)  \ar@<-5ex>[u]^{j}
	\ar[r]^<<<<<{\frown} & \cosheaf{F}_{\ast}(V) \ar[u]^{k} .
	}
	\end{equation*}
For any $\mu \in \sheaf{F}_{\ast + m}(U)$ and $c \in \cosheaf{F}^{m}(V)$, the cap product
satisfies
	\begin{equation}
	\label{eq:naturality}
	k \big( i(\mu) \frown c \big) = \mu \frown j(c).
	\end{equation}
Let $o_U := \cosheaf{O}^{-1}(U \subseteq M)(o)$ and
$o_V := \cosheaf{O}^{-1}(V \subseteq M)(o)$.
The map $f$ induces pull-backs
	\begin{equation*}
	\xymatrix{
	f^m_U : \cosheaf{O} \big ( U \big) \to \cosheaf{F}^m \big( U \big) &&
	f^m_V : \cosheaf{O} \big ( V \big) \to \cosheaf{F}^m \big( V \big).
	}
	\end{equation*}
By Equation \ref{eq:naturality}, the following diagram commutes:
	\begin{equation*}
	\label{eq:cap_square}
	\xymatrix{
	\sheaf{F}_{\ast + m}(U) \ar[d]_{\frown f^m_U(o_U)} \ar[rr]^i 
	&& \sheaf{F}_{\ast + m}(V) \ar[d]^{\frown f^m_V(o_V)} \\
	\cosheaf{F}_\ast(U)  && \cosheaf{F}_\ast(V). \ar[ll]_k
	}
	\end{equation*}
This triple of data, over all $(\Sstrat, \Jcontrol)$-basic opens, 
generates the $\Sstrat$-constructible bisheaf $\bisheaf{F}_{\ast}$.
\end{ex}

\begin{prop}
\label{prop:thom_isomorphism}
Let $f : Y \to M$ be a $(\Sstrat, \Jcontrol)$-constructible map to an oriented $m$-manifold $M$
and $\bisheaf{F}_{\ast}$ its bisheaf as constructed in Example \ref{ex:bisheaf}.
For a top dimensional stratum $S \in \Sstrat$, suppose the restriction $f |_{f^{-1}(S)} : f^{-1}(S) \to S$
is a fiber bundle over $S$.
Then for any $\Sstrat$-basic open $U \subseteq M$ associated to $S$, the cap product
$\Ffunc_\ast(U) : \sheaf{F}_{\ast + m}(U) \to \cosheaf{F}_\ast(U)$ is an isomorphism.
\end{prop}
\begin{proof}
Choose a point $p \in U$.
Since $U$ is contractible, $f |_{f^{-1}(S)}$ is a trivial bundle over $U$.
This means that there is a homeomorphism $h$ that makes the following diagram commute:
	\begin{equation*}
	\xymatrix{
	f^{-1}(U) \ar[rd]^f \ar[rr]^h && f^{-1}(p) \times U \ar[ld]^{\pi_2} \\
	& U. &
	}
	\end{equation*}
The projection $\pi_1 : f^{-1}(p) \times U \to f^{-1}(p)$ is an $m$-disk bundle over $f^{-1}(p)$ putting us
in the setting of the Thom isomorphism.
The element $\frown f^m_U(o_U) \in \cosheaf{F}^m(U)$ is a Thom class making the cap product
	\begin{equation*}
	\xymatrix{
	\sheaf{F}_{\ast + m}(U) \ar[rr]^{\frown f^m_U(o_U)} && \cosheaf{F}(U)
	}
	\end{equation*}
the Thom isomorphism.
Alternatively, we may view $\sheaf{F}_{\ast+m}(U)$ as the homology of the $m$-fold suspension
of $f^{-1}(p)$ making the cap product the suspension isomorphism.
\end{proof}

\begin{defn}
An $\Sstrat$-constructible bisheaf $\bisheaf{I} = \big( \sheaf{I}, \cosheaf{I}, \Ifunc \big)$ 
around $X$ is an \define{isobisheaf} if~$\sheaf{I}$
is an episheaf and $\cosheaf{I}$ is a monocosheaf.
\end{defn}

Let $\Loc(X)$ be the full subcategory of $\Sheaf(X)$
consisting of local systems.
Let $\Coloc(X)$ be the full subcategory of $\Cosheaf(X)$
consisting of colocal systems.
The two categories $\Loc(X)$ and $\Coloc(X)$ are equivalent.
The equivalence takes an $\Sstrat$-constructible local system $\sheaf{F}$ to the $\Sstrat$-constructible
colocal system $\cosheaf{F}$
generated by assigning to every $\Sstrat$-basic open $U$ the group $\sheaf{F}(U)$ and
by assigning to every pair $V \subseteq U$ of $\Sstrat$-basic opens the map $\sheaf{F}^{-1}(V \subseteq U)$.
Similarly, the equivalence takes an $\Sstrat$-constructible colocal system $\cosheaf{G}$ 
to the $\Sstrat$-constructible
local system $\sheaf{G}$ generated by assigning to every $\Sstrat$-basic open $U$ the group $\cosheaf{F}(U)$ and
by assigning to every pair $V \subseteq U$ of $\Sstrat$-basic opens the map $\cosheaf{F}^{-1}(V \subseteq U)$.

\begin{prop}
\label{prop:image_local_system}
Let $\bisheaf{I} = (\sheaf{I}, \cosheaf{I}, \Ifunc)$ be an $\Sstrat$-constructible isobisheaf around $X$.
Then the image $\image \Ifunc$, generated by the images of $\big \{ \Ifunc(U) \big\}$
over all $\Sstrat$-basic opens, 
is a colocal system under $X$ and the coimage $\coimage \Ifunc$, generated by the coimages of
$\big \{ \Ifunc(U) \big\}$ over all $\Sstrat$-basic opens, is a local system over $X$.
Furthermore, $\image \Ifunc$ is equivalent to $\coimage \Ifunc$.
\end{prop}
\begin{proof}
For a pair of $\Sstrat$-basic opens $V \subseteq U$, consider the following
commutative diagram:
	\begin{equation*}
	\xymatrix{
	\sheaf{I}(U) \ar[d]_{\Ifunc(U)} \ar@{->>}[rr]^{\sheaf{I}(V \subseteq U)} && \sheaf{I}(V) \ar[d]^{\Ifunc(V)} \\
	\cosheaf{I}(U) && \cosheaf{I}(V) \ar@{^{(}->}[ll]_{\cosheaf{I}(V \subseteq U)}.
	}
	\end{equation*}

We first show that $\cosheaf{I}(V \subseteq U)$ restricts to an isomorphism
from $\image \Ifunc(V)$ to $\image \Ifunc(U)$.	
For any element $c \in \image \Ifunc(V)$, there is, by surjectivity
of $\sheaf{I}(V \subseteq U)$, an element $a \in \sheaf{I}(U)$ 
such that $c = \Ifunc(V) \circ \sheaf{I}(V \subseteq U)(a)$.
This means that $\cosheaf{I}(V \subseteq U)$ restricts to a homomorphism
from $\image \Ifunc(V)$ to $\image \Ifunc(U)$.
Suppose $d = \Ifunc(U)(a)$.
Then there is a $c \in \cosheaf{I}(V)$,
namely $c = \Ifunc(V) \circ \sheaf{I}(V \subseteq U) \big)(a)$, such that
$d = \cosheaf{I}(V \subseteq U)(c)$.
This means that the restriction is surjective.
By injectivity of $\cosheaf{I}(V \subseteq U)$, the restriction is also injective.
These images and the maps between them
generate the colocal system $\image \Ifunc$.
	
We now show that $\sheaf{I}(V \subseteq U)$ quotients to an isomorphism
from $\coimage \Ifunc(U)$ to $\coimage \Ifunc(V)$.
For any element $a \in \ker \Ifunc(U)$, $\sheaf{I}(V \subseteq U)(a)$ is, by injectivity
of $\cosheaf{I}(V \subseteq U)$, an element of $\ker \Ifunc(V)$.
This means that $\sheaf{I}(V \subseteq U)$ quotients to a homomorphism
from $\coimage \Ifunc(U)$ to $\coimage \Ifunc(V)$.
For any $a \in \sheaf{I}(U)$, if $\sheaf{I}(V \subseteq U)(a)$ is in $\ker \Ifunc(V)$,
then $a \in \ker \Ifunc(U)$.
This means that the quotient map is injective.
By surjectivity of $\sheaf{I}(V \subseteq U)$, the quotient map
is also surjective.
These coimages and the maps between them generate the local system
$\coimage \Ifunc$.

By the first isomorphism theorem, $\image \Ifunc(U)$ is isomorphic
to $\coimage \Ifunc(U)$ for every $U$.
The colocal system $\image \Ifunc$ is therefore equivalent to the local system $\coimage \Ifunc$.
\end{proof}

Let $\bisheaf{F}$ be an $\Sstrat$-constructible bisheaf over $X$.
Epification of $\sheaf{F}$ and monofication of $\cosheaf{F}$ results in an isobisheaf
$\Iso\big( \bisheaf{F} \big) := \Big( \Epi\big( \sheaf{F} \big), 
	\Mono \big( \cosheaf{F} \big), \Iso (\Ffunc) := \lbar \eta \big( \cosheaf{F} \big) \circ 
	\Ffunc \circ \ubar \eta \big( \sheaf{F} \big) \Big);$
see Diagram \ref{dgm:isobisheaf_map}.
Consider a bisheaf map $\ulbar{\alpha} : \bisheaf{F} \to \bisheaf{G}$ in $\Bisheaf(X, \Sstrat)$.
The universal property of episheaves and monocosheaves induces a map of isobisheaves:
	\begin{equation}
	\label{dgm:isobisheaf_map}
	\begin{gathered}
	\xymatrix{
	\Epi \big( \sheaf{F} \big) \ar@{-->}[rr]^{\Epi ( \ubar{\alpha} )} \ar@{^{(}->}[d]_{\ubar{\eta}(\sheaf{F})}
	&& \Epi \big( \sheaf{G} \big) \ar@{^{(}->}[d]^{\ubar{\eta}(\sheaf{G})} \\
	\sheaf{F} \ar[rr]^{\ubar{\alpha}} \ar[d]_{\Ffunc}
	 && \sheaf{G} \ar[d]^{\Gfunc} \\
	\cosheaf{F} \ar@{->>}[d]_{\lbar{\eta}(\cosheaf{F})} && \cosheaf{G} \ar[ll]_{\lbar{\alpha}}  
	\ar@{->>}[d]^{\lbar{\eta}(\cosheaf{G})} \\ 
	\Mono \big( \cosheaf{F} \big) && \Mono \big( \cosheaf{G} \big) 
	\ar@{-->}[ll]_{\Mono ( \lbar{\alpha} )}.
	}
	\end{gathered}
	\end{equation}
Thus the assignment to each bisheaf its isobisheaf is functorial.

\begin{defn}
The \define{isofication} of $\Sstrat$-constructible bisheaves around $X$ is the functor
	$$\Iso : \Bisheaf(X, \Sstrat) \to \Bisheaf(X, \Sstrat)$$
that sends its bisheaf $\bisheaf{F}$ to its isobisheaf $\Iso \big( \bisheaf{F}\big)$.
Let $\ulbar{\eta} = \big( \ubar{\eta}, \lbar{\eta} \big): \id_{\Bisheaf(X, \Sstrat)} \Rightarrow \Iso$ be the
natural transformation induced by $\ubar{\eta}$ and $\lbar{\eta}$.
\end{defn}


\section{\'Etale Opens}
\label{sec:etale}
The idea of an \'etale open was introduced by Grothendieck in algebraic geometry 60 years ago as a 
natural generalization of an open set.  
For us, it is important to have persistent local systems $\mathcal L(U)$ (see Section \ref{sec:introduction}) 
not only for open sets $U$, but for \'etale opens as well.
While it is true that the image
$\image U$ of an \'etale open of $M$ is an open subset of $M$, it is not true that $\mathcal L(\image U)$ contains 
all the information of $\mathcal{L}(U)$.  
In fact $\mathcal L(\image U)$ can vanish while $\mathcal{L}(U)$ is still large.  
If $M = \Rspace$, then every path-connected \'etale open is a path-connected open set. 
This is another way in which  the $1$-dimensional case is much simpler.  

In this section we develop the notion of an \'etale open of a manifold $M$ without boundary.
In the last section, we saw that every constructible bisheaf around $M$ has associated to it 
a local system over $M$.
Now, we pull-back the bisheaf along any \'etale open $a : A \to M$ then use the same procedure to 
compute its persistent local system over $A$.
This gives us our collection of local systems one for every \'etale open of $M$ which constitutes finer information about the bisheaf.  

\begin{defn}
An \define{\'etale open} of $M$ is a continuous map
$a : A \to M$ from a Hausdorff, second countable space $A$ to $M$
that is locally a homeomorphism for every point of $A$.
An \define{\'etale map} $\mu: a \to b$ is a continuous map
$\mu : A \to B$ such that
the following diagram commutes:
	\begin{equation*}
	\xymatrix{
	A \ar[dr]_-{a} \ar[rr]^{\mu} && B \ar[ld]^-{b} \\
	& M. &
	}
	\end{equation*}
Let $\Etal(M)$ be the category of \'etale opens of $M$.
The initial object of $\Etal(M)$ is the empty \'etale open $\emptyset : \emptyset  \to M$
and the terminal object is the identity \'etale open $\id_M : M \to M$.
Note that every open set of $M$ is an \'etale open.
\end{defn}

Given a stratified space $(M, \Sstrat)$ and any \'etale open $a : A \to M$,
the stratification~$\Sstrat$ pulls-back along $a$ to a stratification
$a^\star \Sstrat$ (see Definition \ref{defn:stratifed_space}) of $A$ as follows.
The filtration
$\emptyset \subseteq M_0 \subseteq \cdots \subseteq M_n = M$
that gives rise to $\Sstrat$ lifts to a filtration
$\emptyset \subseteq A_0 \subseteq \cdots \subseteq A_n = A$
where $A_i = a^{-1}(M_i)$.
Every point $a \in A$ has a neighborhood $U$ such that the restriction $a |_U : U \to a(U)$
is a homeomorphism.
Since the point $a(p)$ is locally an open cone over a lower-dimensional stratified space,
the point $p$ is locally an open cone over the same lower-dimensional stratified space.

\begin{defn}
Let $(M, \Sstrat)$ be a stratified manifold.
An \'etale open $a : A \to M$ is \define{$\Sstrat$-constructible} if for 
every stratum $S \in \Sstrat$, $a^{-1}(S)$ is empty or 
the restriction $a |_{a^{-1}(S)} : a^{-1}(S) \to S$ is a covering space.
Let $\Etal(M, \Sstrat)$ be the category of $\Sstrat$-constructible \'etale opens.
\end{defn}
%
%
%
%
%
%


\begin{prop}
\label{prop:con-etale}
Let $(M, \Kstrat)$ be a triangulation of a manifold $M$.
For any \'etale open $a : A \to M$, there is an \'etale map $\mu : a \to b$ to an 
$\Kstrat$-constructible \'etale open $b$ satisfying the following universal property.
For any \'etale map $\nu : a \to c$, where~$c$ is a $\Kstrat$-constructible
\'etale open, there is a unique \'etale map $\eta : b \to c$ that makes the following
diagram commute:
	\begin{equation*}
	\xymatrix{
	a \ar[rd]_{\nu} \ar[r]^{\mu} & b \ar@{-->}[d]^{\eta} \\
	& c.
	}
	\end{equation*}
\end{prop}
\begin{proof}
The map $a$ takes the poset of strata $a^\star \Kstrat$ to the poset of strata $\Kstrat$.
Denote by $\tilde a : a^\star \Kstrat \to \Kstrat$ this poset map.
Let $E := \bigsqcup_{\sigma \in a^\star \Kstrat} \st\; \tilde a(\sigma)$ be the disjoint union of open
stars of strata in~$\Kstrat$
and denote by $e : E \to M$ the map that takes each point $x \in E$ to its original copy in $M$.
Note that there may be many strata in $a^\star \Kstrat$ that map to a single stratum $\tau \in \Kstrat$.
In this case, $E$ has $n$ copies of $\st\; \tau$ where $n$ is the number of strata in $a^\star \Kstrat$
that map to~$\tau$.
We now glue together the pieces of $E$.
Consider two points $x, y \in E$ and suppose $x \in \sigma_1$ and $y \in \sigma_2$
for strata $\sigma_1, \sigma_2 \in a^\star \Kstrat$.
Let us say $x$ and $y$ are related, $x \sim y$, if $e(x) = e(y)$ and $\sigma_1 \leq \sigma_2$
or $\sigma_2 \leq \sigma_1$.
The relation $\sim$ is reflexive and symmetric but not transitive.
Take the transitive closure of $\sim$, let $D := \sfrac{E}{\sim}$ be the quotient space,
and let $d : D \to M$ the quotient of $e$.
Note that every point $x \in D$ has an open neighborhood $U$ such that the restriction
$d |_U : U \to d(U)$ is a homeomorphism.
In other words, $d$ is a local homeomorphism.
Furthermore, for every stratum $\tau \in \Kstrat$, $d^{-1}(\tau)$ is either empty or a covering space
over $\tau$.
However, $D$ may not be Hausdorff; see the example in Figure \ref{fig:constructible}.
We remedy this problem by taking a second quotient as follows.
Let us say two points $p ,q \in D$ are related, $p \approx q$, if there is a path
$\gamma : [0,1] \to D$ satisfying the following two properties: $\gamma(0) = p$ and $\gamma(1) = q$,
and $d \circ \gamma(t) = d \circ \gamma(1- t)$ for all $0 \leq t \leq \sfrac{1}{2}$.
The relation $\approx$ is reflexive and symmetric but not transitive.
Take the transitive closure of $\approx$, let $B := \sfrac{D}{\approx}$, 
and let $b : B \to M$ be the induced continuous map.
The map $b$ is a $\Kstrat$-constructible \'etale open.
Let $\mu : A \to B$ be the composition of the continuous inclusion $A \to D$ followed by the quotient map $D \to B$.

We now prove the universal property.
Suppose $c$ is a $\Kstrat$-constructible \'etale open and $\nu : a \to c$ is given.
For every stratum $\tau \in b^\star \Kstrat$, $\mu^{-1}(\tau)$ is non-empty, by construction of $b$.
Furthermore, every stratum of $a^\star \Kstrat$ contained in $\mu^{-1}(\tau)$ maps, via $\nu$, to a single stratum in $c^\star \Kstrat$
because otherwise, $C$ could not be Hausdorff.
Let $\eta : B \to C$ be the unique map induced by sending every stratum $\tau \in b^\star \Kstrat$ to the stratum
$\nu \circ \mu^{-1}(\tau)$.
For any open set $U \subseteq C$, $\nu^{-1}(U)$ is open because $\nu$ is continuous.
Since $\mu$ is the composition of an open inclusion followed by a quotient map, $\mu$ maps
$\nu^{-1}(U)$ to an open set.
This makes $\eta$ continuous.
\end{proof}

\begin{ex}
Consider the triangulation $\Kstrat$ of $\Rspace^2$ illustrated in Figure~\ref{fig:constructible}.
Let $a : A \to \Rspace^2$ be the \'etale open in red.
Then its universal $\Kstrat$-constructible \'etale open $b : B \to \Rspace^2$ is in yellow.

\begin{figure}
\centering
\includegraphics{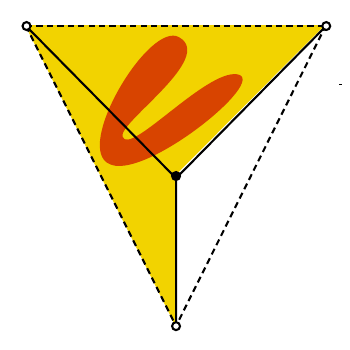}
\caption{Consider the triangulation $\Kstrat$ of the plane $\Rspace^2$ above.
The yellow \'etale open is the universal $\Kstrat$-constructible \'etale open containing the red \'etale open.}
\label{fig:constructible}
\end{figure}
\end{ex}

Given an $\Sstrat$-constructible bisheaf $\bisheaf{F} = (\sheaf{F}, \cosheaf{F}, \Ffunc)$ around $M$ and 
an \'etale
open $a : A \to M$, we may pull $\bisheaf{F}$ back to a $a^\star \Sstrat$-constructible
bisheaf $a^\star \bisheaf{F} = (a^\star \sheaf{F}, a^\star \cosheaf{F}, a^\star \Ffunc)$ around $A$ as follows.
Choose a subposet $\Acat \subseteq \Basic(A, a^\star \Sstrat)$ such that $\Acat$ is a basis for the topology
on $A$ and $a(J)$ is an $\Sstrat$-basic open for all $J \in \Acat$.
Let $\Gfunc : \Acat \to \Ab$ be the contravariant functor that assigns to
each open set $J$ the object $\Gfunc(J) := \sheaf{F}(a(J))$ and to every pair
$J \subseteq I$ the morphism $\sheaf{F}(a(J) \subseteq a(I))$.
Let $\Hfunc : \Acat \to \Ab$ be the covariant functor that assigns to each
open set $J$ the object $\cosheaf{F}(a(J))$ and to every pair $J \subseteq I$ the morphism
$\cosheaf{F}(a(J) \subseteq a(I))$.
Let $\big \{ \Qfunc(J) \big \}_{J \in \Acat}$ be the set of maps 
$\Qfunc(J) : \Gfunc(J) \to \Hfunc(J)$ where $\Qfunc(J) := \Ffunc(a(J))$.
This data, as discussed in Section~\ref{sec:bisheaves}, generates the bisheaf $a^\star \bisheaf{F}$.
For a morphism $\mu : a \to b$ of \'etale opens, the two bisheaves $a^\star \bisheaf{F}$
and $\mu^\star b^\star \bisheaf{F}$ around $A$ are isomorphic.

\section{Isobisheaf Stacks}
\label{sec:stacks}
We finally get to the central construction of this paper: isobisheaf stacks.  
Given a constructible bisheaf over a manifold $M$, we now have a local system
for each \'etale open of $M$.
Here we assemble these local systems into a stack.  The advantage is that the isobisheaf stack has good functorial properties which are useful, for example, in proving stability.  
The whole construction of the persistent local systems can be thought of this way:
$$\genfrac\{\}{0pt}0{\mbox{Maps }}{\mbox{} X\to M }
\longrightarrow 
\genfrac\{\}{0pt}0{\mbox{Bisheaves}}{\mbox{around } M } 
\longrightarrow 
\genfrac\{\}{0pt}0{\mbox{Isobisheaf stacks}}{\mbox{ around } M }\longrightarrow 
\genfrac\{\}{0pt}0{\mbox{Local systems for}}{\mbox{each \'etale open of }M} .
$$
\vspace{1em}

\begin{defn} \label{defn:isobisheaf_stack}
Let $(M, \Sstrat)$ be a stratified manifold.
An \define{$\Sstrat$-constructible isobisheaf stack} $\Fstack$ \emph{around} $M$ 
is the assignment to $\Etal(M)$ the following data satisfying the following
axiom:
	\begin{itemize}
	
	\item To each \'etale open $a: A \to M$, 
	$\Fstack(a)$ is an $a^\star \Sstrat$-constructible isobisheaf 
	$\big( \sheaf{F}_a, \cosheaf{F}_a, \Ffunc_a \big)$.
	
	\item To each \'etale map $\mu : a \to b$, $\Fstack(\mu) : 
	\mu^\star \Fstack(b) \to \Fstack(a)$ is a bisheaf map
		\begin{equation}
		\label{dgm:stack}
		\begin{gathered}
		\xymatrix{
		\mu^\star \sheaf{F}_b \ar[d]_{\mu^\star \Ffunc_b} \ar@{^{(}->}[rr]^{\uFstack(\mu)} 
		&& \sheaf{F}_a \ar[d]^{\Ffunc_a} \\
		\mu^\star \cosheaf{F}_b && \cosheaf{F}_a \ar@{->>}[ll]^{\lFstack(\mu)}
		}
		\end{gathered} 
		\end{equation}
	where $\uFstack(\mu)$ is injective and $\lFstack(\mu)$ is surjective for all open sets.
	
	\item For each pair of \'etale maps $\mu : a \to b$ and $\nu : b \to c$,
	$\Fstack(\nu \circ \mu) = \mu^\star \Fstack(\nu) \circ \Fstack(\mu)$.
	
	\end{itemize}
By Proposition \ref{prop:image_local_system}, the colocal system $\image \Ffunc_a$ under $M$
is equivalent to the local system $\coimage \Fstack(a)$ over $M$.
Let $\image \Fstack(a) := \image \Ffunc_a$ be the \define{persistent colocal system} of $\Fstack$ under $a$
and let $\coimage \Fstack(a) := \coimage \Fstack(a)$ be the \define{persistent local system} of $\Fstack$ over $a$.

Let $\Fstack$ be an $\Sstrat$-constructible isobisheaf around $M$ and $\Gstack$
a $\Tstrat$-constructible isobisheaf around $M$ where $\Sstrat$ and $\Tstrat$ may be different.
A \define{map of constructible isobisheaf stacks} $\ulbar \Phi : \Fstack \to \Gstack$
is the following data satisfying the following axiom:
	\begin{itemize}
	\item To each \'etale open $a : A \to M$, $\ulbar \Phi(a) : \Fstack(a) \to \Gstack(a)$
	is a bisheaf map
		\begin{equation}
		\label{dgm:stack_map}
		\begin{gathered}
		\xymatrix{
		\sheaf{F}_a \ar[d]_{\Ffunc_a} \ar[rr]^{\ubar \Phi(a)} 
		&& \sheaf{G}_a \ar[d]^{\Gfunc_a} \\
		\cosheaf{F}_a && \cosheaf{G}_a \ar[ll]^{\lbar \Phi(a)}.
		}
		\end{gathered}
		\end{equation}
	Note there are no conditions on $\ubar \Phi(a)$ and $\lbar \Phi(a)$ other than
	that the diagram commutes for every open set of $A$.
	
	\item For each \'etale map $\mu : a \to b$, the following diagram commutes
	for every open set of $A$:
	
	\begin{equation*}
	\label{dgm:stack_map_map}
	\begin{gathered}
	\xymatrix{
	\mu^\star \sheaf{F}_b \ar[ddd]_{\mu^\star \Ffunc_b} \ar[rrrr]^{\mu^\star \ubar \Phi(b)} 
	\ar@{^{(}->}[rd]^{\uFstack(\mu)}
	&&&& \mu^\star \sheaf{G}_b \ar[ddd]^{\mu^\star \Gfunc_b} 
	\ar@{^{(}->}[ld]^{\uGstack(\mu)} \\
	& \sheaf{F}_a \ar[d]_{\Ffunc_a} \ar[rr]^{ \ubar \Phi(a)} 
		&& \sheaf{G}_a \ar[d]^{\Gfunc_a}  & \\
	& \cosheaf{F}_a \ar@{->>}[ld]^{ \lFstack(\mu)} 
	&& \cosheaf{G}_a \ar[ll]^{\lbar \Phi(a)} 
	\ar@{->>}[rd]^{\lGstack(\mu)} & \\
	\mu^\star \cosheaf{F}_b &&&& \mu^\star \cosheaf{G}_b \ar[llll]^{\mu^\star  \lbar \Phi(b)}.
	}
	\end{gathered}
	\end{equation*}
	\end{itemize}
Let $\Stack(X)$ be the category consisting of isobisheaf stacks around $M$, each constructible
with respect to some stratification of $M$, and constructible isobisheaf stack maps.
\end{defn}

Given an $\Sstrat$-constructible isobisheaf stack $\Fstack$ around $M$, we have a persistent
colocal system $\image \Fstack(a)$ for each \'etale open $a : A \to M$.
For an \'etale map $\mu : a \to b$, the two colocal systems $\image \Fstack(a)$
and $\image \Fstack(b)$ are related by Diagram \ref{dgm:stack}.
Let $\cosheaf{I} := \image \big(\Ffunc_a \circ \uFstack(\mu) \big)$ and
$\cosheaf{K} := \cosheaf{I} \cap \ker \lFstack(\mu)$.
Then 
	\begin{equation*}
	\xymatrix{
	\image \mu^\star \Fstack(b) & \ar@{->>}[l]_-{/ \cosheaf{K}} \cosheaf{I}
	\ar@{^{(}->}[r] & \image \Fstack(a).
	}
	\end{equation*}
In other words, the data $\image \mu^\star \Fstack(b)$ \emph{persists} in $\image \Fstack(a)$
as a quotient of a sub-colocal system.
Given a stack map $\Phi : \Fstack \to \Gstack$ and an \'etale open $a$,
the two colocal systems $\image \Fstack(a)$ and $\image \Gstack(a)$ are related
by Diagram \ref{dgm:stack_map}.
Thus
	\begin{equation}
	\label{eq:subquotient}
	\xymatrix{
	\image \Fstack(a) & \ar@{->>}[l]_-{/ \cosheaf{K}} \cosheaf{I}
	\ar@{^{(}->}[r] & \image  \Gstack(a).
	}
	\end{equation}
where $\cosheaf{K}$ and $\cosheaf{I}$ are defined similarly.
As we will see in Section \ref{sec:stability}, this observation implies that persistent
colocal systems satisfy the property \emph{Stability} of Section \ref{sec:introduction}.

\begin{ex}
\label{ex:stack_example}
A constructible bisheaf $\bisheaf{F}$ over a manifold $M$ gives rise to a constructible
isobisheaf stack $\Fstack$ as follows.
For each \'etale open $a :A \to M$, let $\Fstack(a) := \Iso \big( a^\star \bisheaf{F} \big)$.
For an \'etale map $\mu : a \to b$, we have the following commutative diagram where the top and bottom horizontal
maps are induced by the universal property of $\Epi$ and $\Mono$ respectively:
	\begin{equation}
	\label{dgm:quillen}
	\begin{gathered}
	\xymatrix{
	\mu^\star \Epi \big( b^\star \sheaf{F} \big) \ar@{^{(}->}[rr]^{\ubar{\beta}} 
	\ar@{^{(}->}[d]_{\mu^\star \ubar \eta ( b^\star \sheaf{F} ) } 
	&& \Epi \big( a^\star \sheaf{F} \big) \ar@{^{(}->}[d]^{\ubar \eta ( a^\star \Ffunc )} \\
	\mu^\star b^\star \sheaf{F}  \ar[rr]^{\cong} 
	\ar[d]_{\mu^\star  b^\star \Ffunc } 
	&& a^\star \sheaf{F}  \ar[d]^{a^\star \Ffunc} \\
	\mu^\star b^\star \cosheaf{F}   
	\ar[d]_{\mu^\star \lbar \eta (b^\star \cosheaf{F}) } 
	&& \ar[ll]^{\cong} a^\star \cosheaf{F}  \ar[d]^{\lbar \eta (a^\star \cosheaf{F}) )} \\
	\mu^\star \Mono \big( b^\star \cosheaf{F} \big) && \ar@{->>}[ll]^{\lbar{\beta}} \Mono 
	\big( a^\star \cosheaf{F} \big).
	}
	\end{gathered}
	\end{equation}
Both $\mu^\star \Epi \big( b^\star \sheaf{F} \big)$ and $\Epi \big( a^\star \sheaf{F} \big)$
are sub-episheaves of $\mu^\star b^\star \sheaf{F} \cong a^\star \sheaf{F}$ and the latter
is maximal.
This makes the sheaf map $\ubar \beta$ injective on every open set.
Both  $\mu^\star \Mono \big( b^\star \cosheaf{F} \big)$ and $\Mono \big( a^\star \cosheaf{F} \big)$
are quotient-monocosheaves of $\mu^\star b^\star \cosheaf{F} \cong a^\star \cosheaf{F}$
and the later is minimal.
This makes the cosheaf map $\lbar \beta$ surjective on every open set.
Let $\Fstack(\mu) := \big( \ubar{\beta}, \lbar{\beta} \big)$.

Note that if the bisheaf $\bisheaf{F}$ is constructed from a constructible map $f$ as in Example~\ref{ex:bisheaf},
then every persistent colocal system is a subquotient of the homology of the fibers of $f$.
In other words, persistent colocal systems satisfy property \emph{Relation to homology of fibers}
of Section \ref{sec:introduction}.

A bisheaf map $\ulbar \alpha : \bisheaf{F} \to \bisheaf{G}$ gives rise
to a map of isobisheaf stacks $\ulbar \Phi : \Fstack \to \Gstack$ as follows.
Recall Diagram \ref{dgm:isobisheaf_map}.
For each \'etale open $a : A \to M$, 
replace the bisheaf $\bisheaf{F}$ in Diagram \ref{dgm:isobisheaf_map}
with $a^\star \bisheaf{F}$ and replace the bisheaf $\bisheaf{G}$
with $a^\star \bisheaf{G}$ to get the following diagram:
	\begin{equation*}
	\begin{gathered}
	\xymatrix{
	\Epi \big( a^\star \sheaf{F} \big) \ar[rr]^{\Epi ( a^\star  \ubar{\alpha} )} 
	\ar@{^{(}->}[d]_{\ubar{\eta}(a^\star \sheaf{F})}
	&& \Epi \big( a^\star \sheaf{G} \big) \ar@{^{(}->}[d]^{\ubar{\eta}(a^\star \sheaf{G})} \\
	a^\star \sheaf{F} \ar[rr]^{a^\star  \ubar{\alpha}} \ar[d]_{a^\star \Ffunc}
	 && a^\star \sheaf{G} \ar[d]^{a^\star \Gfunc} \\
	a^\star \cosheaf{F} \ar@{->>}[d]_{\lbar{\eta}(a^\star \cosheaf{F})} && a^\star \cosheaf{G} \ar[ll]_{a^\star \lbar{\alpha}}  
	\ar@{->>}[d]^{\lbar{\eta}(a^\star \cosheaf{G})} \\ 
	\Mono \big(a^\star \cosheaf{F} \big) && \Mono \big(a^\star \cosheaf{G} \big) 
	\ar[ll]_{\Mono ( a^\star  \lbar{\alpha} )}.
	}
	\end{gathered}
	\end{equation*}
For each \'etale open $a : A \to M$, let $\ubar \Phi(a) := \Epi(a^\star \ubar \alpha)$ and let
$\lbar \Phi(a) := \Mono(a^\star \ubar \alpha)$.
\end{ex}

\begin{prop}
\label{prop:constructible_stack}
Let $(M, \Kstrat)$ be a triangulation of a manifold $M$,
$\bisheaf{F}$ an $\Kstrat$-constructible bisheaf over $M$, and
$\Fstack$ its $\Kstrat$-constructible isobisheaf stack.
For an \'etale open $a : A \to M$, let $\mu : a \to b$
be the universal \'etale map to a $\Kstrat$-constructible \'etale open $b$
in the sense of Proposition \ref{prop:con-etale}.
Then $\Fstack (a) \cong \mu^\star \Fstack (b)$.
\end{prop}

\begin{proof}
Since every stratum of $\Kstrat$ is contractible, every stratum of $a^\star \Kstrat$
is also contractible.
The bisheaf $a^\star \bisheaf{F}$ is completely determined (up to an isomorphism) by the assignment to each
stratum $\sigma \in a^\star \Kstrat$ the map 
$\bisheaf{F}(\st\; a(\sigma )) : \sheaf{F}(\st\; a(\sigma)) \to \cosheaf{F}(\st\; a(\sigma))$
and to each relation $\sigma \leq \tau$ the bisheaf map 
$\big( \sheaf{F}( \st\; a(\tau) \subseteq \st\; a(\sigma) ) , \cosheaf{F}( \st\; a(\tau) \subseteq \st\; a(\sigma) \big)$.
For convenience, we will simply identity $a^\star \bisheaf{F}$ with this stratum-wise assignment.
The bisheaf $b^\star \bisheaf{F}$ is determined similarly since every stratum of $b^\star \Kstrat$ is contractible.
Consider Diagram \ref{dgm:quillen}.
We prove the claim by showing that $\ubar \beta$ is surjective and $\lbar \beta$ is injective.

For very stratum $\sigma \in b^\star \Kstrat$, there is, by universality of $b$, at least one
stratum $\sigma' \in a^\star \Kstrat$ such that $\mu(\sigma') = \sigma$.
Consider the following sub-episheaf $\sheaf{E}$ of $b^\star \sheaf{F}$.
For each stratum $\sigma \in b^\star \Kstrat$, let
	$$ \sheaf{E}(\st\; \sigma) := \biguplus_{\{\sigma' \in a^\star \Kstrat \; | \; 
	\mu(\sigma') \subseteq \sigma \}} \Epi \big( a^\star \sheaf{F} \big) (\st\; \sigma') \subseteq 
	\sheaf{F} \big( \st\; b(\sigma) \big)$$
and for each each $\sigma \leq \tau$, let 
$ \sheaf{E}(\st\; \tau \subseteq \st\; \sigma) : \sheaf{E}(\st\; \sigma) \to \sheaf{E}(\st\; \tau)$
be the restriction of the map $\sheaf{F}\big( \st\; b(\tau) \subseteq \st\; b(\sigma) \big)$.
Note that $\sheaf{E}(\st\; \tau \subseteq \st\; \sigma)$ is surjective.
The pull-back $\mu^\star \sheaf{E}$ is a sub-episheaf of $a^\star \sheaf{F}$ containing 
$\Epi \big( a^\star \sheaf{F} \big)$.
By maximality of $\Epi \big( a^\star \sheaf{F} \big)$, $\ubar \beta$ is surjective.

The dual argument shows that $\lbar \beta$ is injective.
Consider the following quotient-monocosheaf $\cosheaf{M}$ of $b^\star \cosheaf{F}$.
For each stratum $\sigma \in b^\star \Kstrat$, let
	$$ \cosheaf{M}(\st\; \sigma) := \dfrac{\cosheaf{F}\big(\st\; b(\sigma) \big)}
	{\bigcap_{\{\sigma' \in a^\star \Kstrat \; | \; 
	\mu(\sigma') \subseteq \sigma \}} \ker \lbar \eta (a^\star \cosheaf{F})(\st\; \sigma')}$$
and for each each $\sigma \leq \tau$, let 
$ \cosheaf{M}(\st\; \tau \subseteq \st\; \sigma) : \cosheaf{M}(\st\; \tau) \to \cosheaf{M}(\st\; \sigma)$
be the map $\cosheaf{F}\big( \st\; b(\tau) \subseteq \st\; b(\sigma) \big)$ quotient the intersection of the kernels.
Note that $\cosheaf{M}(\st\; \tau \subseteq \st\; \sigma)$ is injective.
The pull-back $\mu^\star \cosheaf{M}$ is a quotient-monocosheaf of $a^\star \cosheaf{F}$ smaller than 
$\Mono \big( a^\star \cosheaf{F} \big)$.
By minimality of $\Mono \big( a^\star \cosheaf{F} \big)$, $\lbar \beta$ is surjective.
\end{proof}

\section{Dilation}
\label{sec:dilation}
In this section, we begin the task of proving stability of the isobisheaf stack of a map.
Dilation is an operation that coarsens or smooths the data of a constructible bisheaf.

Let $K$ be a simplicial complex.
The \emph{first subdivision} of $K$ is the simplicial complex
$K^1$
whose (open) simplices are chains
$ [\sigma_{i_0} \leq \cdots \leq \sigma_{i_n} ]$
of simplices in $K$.
The face relation 
$[\sigma_{i_0} \leq \cdots \leq \sigma_{i_n}] \leq [\sigma_{j_0} \leq \cdots \leq \sigma_{j_m}]$
in $K^1$ is the subchain relation.
Similarly, the \emph{second subdivision} of $K$ is the simplicial complex
$K^2$ whose (open) simplices are chains
	$$\Big[ [\sigma_{i_0} \leq \cdots \leq \sigma_{i_n}] \leq \cdots \leq [\sigma_{j_0} \leq \cdots 
	\leq \sigma_{j_m}] \Big]$$
of simplices in $K^1$.
The face relation in $K^2$ is the subchain relation.

\begin{defn}
The \define{dilation} of a simplicial complex $K$ is the simplicial map
$\Sigma : K^2 \to K^1$ defined by sending each vertex
$\Big[ [\sigma_{i_0} \leq \cdots \leq \sigma_{i_n}] \Big ] \in K^2$
to the vertex $[\sigma_{i_0}] \in K^1$.
Thus each simplex
$$\Big[ [\sigma_{i_0} \leq \cdots \leq \sigma_{i_l}] \leq 
\cdots \leq [\sigma_{j_0} \leq \cdots \leq \sigma_{j_m}] \leq \cdots \leq
[\sigma_{k_0} \leq \cdots \leq \sigma_{k_n}] \Big] \in K^2$$
maps to the simplex
$ [ \sigma_{k_0} \leq \cdots \leq \sigma_{j_0} \leq \cdots \leq \sigma_{i_0} ] \in K^1.$
Note that for a simplex $\tau \in K$, 
	$$\Sigma^{-1}\big(  [ \tau ] \big) = \cl\; \st\; \big[ [ \tau ] \big] - 
	\bigcup_{\sigma < \tau} \big \{ \cl \; \st \; \big[ [ \sigma ] \big] \big \}.$$
Here $\cl\; \st\; \big[ [ \tau ] \big]$ means the closure of the open star
of $\big[ [ \tau ] \big]$ in $K^2$.
See Figure \ref{fig:dilation}.
\end{defn}

\begin{figure}
\centering
\includegraphics[scale=1.5]{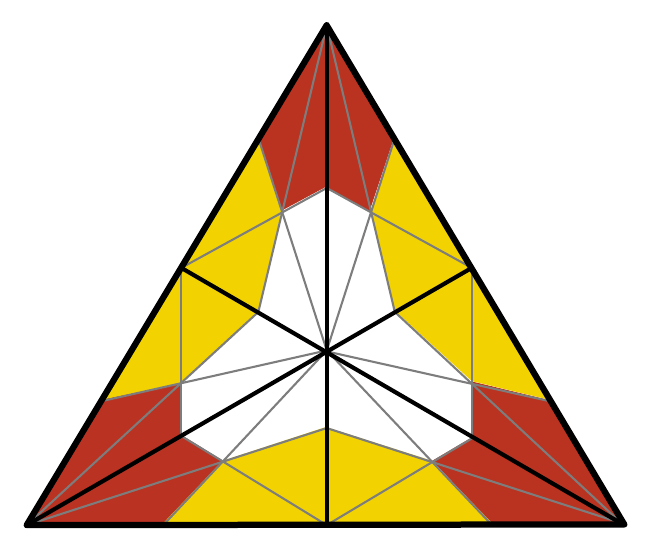}
\caption{The dilation of the $2$-simplex $K$ is the simplicial map $\Sigma : K^2 \to K^1$
where $K^1$ is its first subdivision and $K^2$ its second subdivision.
For each vertex (resp.\ edge) $\tau \in K$, $\Sigma^{-1}\big(  [ \tau ] \big)$ is colored in red (resp.\ yellow).}
\label{fig:dilation}
\end{figure}

Let $(X, \Kstrat)$ be a triangulation and $\phi : | K - K_0 | \to X$ the associated
homeomorphism from a simplicial pair $(K, K_0)$.
We subdivide $(X, \Kstrat)$ by subdividing $(K, K_0)$ and pushing-forward
along $\phi$.
Denote by $(X, \Kstrat^i)$ the $i$-th subdivision of $(X, \Kstrat)$.
The simplicial dilation map $\Sigma : K^2 \to K^1$
gives rise to a continuous dilation map $\Sigma : (X, \Kstrat^2) \to (X, \Kstrat^1)$.

A $\Kstrat^1$-constructible bisheaf $\bisheaf{F}$ around a manifold $M$ pulls back
along the dilation map to produce a $\Kstrat^2$-constructible
bisheaf $\Sigma^\star \bisheaf{F}$ around $M$ generated by the following assignments.
Consider any pair $V \subseteq U$ of $\Kstrat^2$-basic opens.
Suppose
$V$ is associated to a stratum $\tau$ and $U$ to a stratum $\sigma$.
Then $\sigma \leq \tau$ and therefore $\Sigma(\sigma) \leq \Sigma(\tau)$.
The $\Kstrat^2$-constructible sheaf $\Sigma^\star \sheaf{F}$ is generated by 
$\Sigma^\star \sheaf{F}(U) := \sheaf{F}\big( \st\; \Sigma(\sigma) \big)$
and 
$\Sigma^\star \sheaf{F}(V \subseteq U)$ is the map 
$\sheaf{F} \big( \st \; \Sigma(\tau) \subseteq \st\; \Sigma(\sigma) \big).$
The $\Kstrat^2$-constructible cosheaf $\Sigma^\star \cosheaf{F}$ is generated by
$\Sigma^\star \cosheaf{F}(U) := \cosheaf{F}\big( \st\; \Sigma(\sigma) \big)$
and 
$\Sigma^\star \cosheaf{F}(V \subseteq U)$ is the map 
$\cosheaf{F} \big( \st \; \Sigma(\tau) \subseteq \st\; \Sigma(\sigma) \big).$
Let $\Sigma^\star \Ffunc(U) := \Ffunc \big( \st\; \Sigma(\sigma) \big).$

\begin{prop}
\label{prop:dilation_bisheaf_map}
Let $(M, \Kstrat)$ be a triangulation and $\bisheaf{F}$ a $\Kstrat^1$-constructible
bisheaf.
Then there is a canonical bisheaf map $\ulbar \alpha : \Sigma^\star \bisheaf{F} \to \bisheaf{F}$.
\end{prop}
\begin{proof}
By the above construction of $\Sigma^\star \bisheaf{F}$, it is enough to specify
$\ulbar \alpha$ on the open stars of each stratum of $\Kstrat^2$.
For each $\tau \in \Kstrat^2$, we have $\st\; \tau \subseteq \st\; \Sigma(\tau)$.
Since $\Sigma^\star \sheaf{F} ( \st\; \tau)$ is canonically isomorphic to 
$\sheaf{F}\big( \st\; \Sigma(\tau) \big)$, let $\ubar{\alpha}( \st\; \tau )$ be the map
generated by the map $\sheaf{F}\big( \st\; \tau \subseteq  \st\; \Sigma(\tau) \big)$.
Since $\Sigma^\star \cosheaf{F} ( \st\; \tau)$ is canonically isomorphic to $\cosheaf{F}\big( \st\; \Sigma(\tau) \big)$,
let $\lbar{\alpha}( \st\; \tau )$ be generated by the map 
$\cosheaf{F}\big( \st\; \tau \subseteq \st\; \Sigma(\tau) \big)$.
\end{proof}

\begin{defn}
Let $(M, \Kstrat)$ be a triangulation of a manifold 
and $a : A \to M$ a $\Kstrat$-constructible \'etale open.
The \define{shrinking} of $a$ is the $\Kstrat^2$-constructible \'etale open
$\dot a : \dot A \to M$ that is obtained by pulling back $a$ along the continuous dilation map 
$\Sigma : (M, \Kstrat^2) \to (M, \Kstrat^1)$:
	\begin{equation*}
	\xymatrix{
	\dot A \ar@{-->}[d]^{\dot a} \ar@{-->}[rr]^{\mu} && A \ar[d]^{a} \\
	M \ar[rr]^{\Sigma} && M.
	}
	\end{equation*}
For example, suppose $M$ is $2$-dimensional. 
If $a$ is a $2$-stratum of $\Kstrat$ as in Figure \ref{fig:dilation}, 
then $\dot a$ is the white region in the interior.
%
%
\end{defn}

\begin{prop}
\label{prop:shrinking}
Let $(M, \Kstrat)$ be a triangulation of a manifold, $\bisheaf{F}$ a $\Kstrat$-constructible bisheaf,
$a : A \to M$ a $\Kstrat$-constructible \'etale open, and $\mu : \dot a \to a$ the
canonical \'etale map from the shrinking of $a$.
Then the two persistent colocal systems $\image \mu^\star \Iso \big( a^\star \bisheaf{F} \big)$ 
and $\image \Iso \big( \dot a^\star \Sigma^\star \bisheaf{F} \big)$ under $\dot A$ are isomorphic.
\end{prop}
\begin{proof}
The dilation map $\Sigma : (M, \Kstrat^2) \to (M, \Kstrat^1)$ pulls-back to a surjective
$a^\star \Kstrat^1$-constructible map $\Lambda : (\dot A, \dot a^\star \Kstrat^2) \to (A, a^\star \Kstrat^1)$.
The isobisheaf
	$$\Epi \big( \dot a^\star \Sigma^\star \sheaf{F} \big) \mono
	\dot a^\star \Sigma^\star \sheaf{F} \to \dot a^\star \Sigma^\star \cosheaf{F}
	\epi \Mono \big( \dot a^\star \Sigma^\star \cosheaf{F} \big)$$
is the pull-back along $\Lambda$ of the isobisheaf
	$$\Epi \big(a^\star \sheaf{F} \big) \mono
	a^\star \sheaf{F} \to a^\star \cosheaf{F}
	\epi \Mono \big( a^\star \cosheaf{F} \big).$$
For each simplex $\sigma \in \dot a^\star \Kstrat^2$,
$\Lambda(\st\; \sigma) \supseteq \mu(\st \; \sigma)$.
Thus we have the following diagram
	\begin{equation*}
	\xymatrix{
	\Epi \big( \dot a^\star \Sigma^\star \sheaf{F} \big)(\st\; \sigma) 
	\ar[d] \ar@{^{(}->}[rr] && \Epi \big(a^\star \sheaf{F} \big)\big( \mu(\st \; \sigma) \big) \ar[d] \\
	\Mono \big( \dot a^\star \Sigma^\star \cosheaf{F} \big)(\st\; \sigma) && 
	\Mono \big( a^\star \cosheaf{F} \big)\big( \mu(\st \; \sigma) \big) \ar@{->>}[ll]
	}
	\end{equation*}
which induces an isomorphism between the two vertical images.
Therefore $\image \mu^\star \Iso \big( a^\star \bisheaf{F} \big)$ and 
$\image \Iso \big( \dot a^\star \Sigma^\star \bisheaf{F} \big)$ are isomorphic.
\end{proof}

\section{Stability}
\label{sec:stability}
Let $M$ be a compact oriented $m$-manifold
and $\mathfrak{W}(X, M)$ the set 
of all constructible maps $X \to M$ as in Definition \ref{defn:constructible_map}.
For each open set $U \subseteq X \times M$, let
$$T_U := \big\{ f \in \mathfrak{W}(X, M) \; 
\big |\; \mathsf{graph}(f) \subseteq U \big \} .$$
The collection $\big\{ T_U \big \}$ over all open sets $U$ forms the basis for 
the \emph{Whitney topology} on~$\mathfrak{W}(X, M)$.

\begin{thm}
\label{thm:main}
Every map $f \in \mathfrak{W}(X, M)$ has an open neighborhood $\mathcal{U} \subseteq \mathfrak{W}(X, M)$
such that for every map $g \in \mathcal{U}$, their bisheaves $\bisheaf{F}_\ast$ and $\bisheaf{G}_\ast$ are related by canonical bisheaf maps in~$\Bisheaf(M)$:
$$\bisheaf{F}_\ast \leftarrow \Sigma^\star \bisheaf{F}_\ast \to \bisheaf{G}_\ast.$$
\end{thm}

Note that $f$ and $g$ need not be constructible with
respect to the same stratification and therefore
$\bisheaf{F}_\ast$ and $\bisheaf{G}_\ast$
may not be constructible with respect to the same stratification.
Recall $\Sigma : M \to M$ is the dilation map with respect to some triangulation of $M$ and
$\Bisheaf(M)$ is the category of all constructible
bisheaves over $M$.

\begin{proof}
Suppose $f$ is $(\Sstrat, \Jcontrol)$-constructible making $\bisheaf{F}_\ast$
an $\Sstrat$-constructible bisheaf.
By Proposition \ref{prop:triangulation},
there is a triangulation $(M, \Kstrat)$ of $(M, \Sstrat, \Jcontrol)$ and this triangulation
can be chosen so that the open star of each stratum in $\Kstrat$
is contained in an $(\Sstrat, \Jcontrol)$-basic open.
This makes $\bisheaf{F}_\ast$ a $\Kstrat^1$-constructible bisheaf
and $\Sigma^\star \bisheaf{F}$ a $\Kstrat^2$-constructible bisheaf.
The bisheaf map $\Sigma^\star \bisheaf{F}_\ast \to \bisheaf{F}_\ast$ follows from
Proposition \ref{prop:dilation_bisheaf_map}.

Every second-countable, Hausdorff space is metrizable.
Choose a metric on $M$.
For each stratum $\sigma \in \Kstrat$, we have
$\st\; \big[ [ \sigma] \big]  \subseteq \st\; \sigma.$
By compactness of $M$, $\Kstrat$ is finite.
Let	
	\begin{equation}
	\label{eq:injectivity}
	\rho := \min_{\sigma \in \Kstrat} \Haus \big( \st\; \big[ [\sigma] \big] ,
	\st \; \sigma \big)
	\end{equation}
where $\Haus$ is the Hausdorff distance between the two sets.
The set 
	\begin{equation}
	\label{eq:open_set}
	U := \Big \{ f' \in \mathfrak{W}(X, M) \Big | \sup_{x \in X} \Dist \big( f(x), f'(x) \big) < \rho \Big \}
	\end{equation}
is an open neighborhood of $f$ in $\mathfrak{W}(X, M)$. 

Choose a map $g \in \mathcal{U}$ and suppose it is $(\Sstrat', \Jcontrol')$-constructible
making $\bisheaf{G}_\ast$ an $\Sstrat'$-constructible bisheaf.
Choose a triangulation $(M, \Lstrat)$ of $(M, \Sstrat', \Jcontrol')$.
For each $\tau \in \Lstrat$, we assume there is a 
$$\sigma = \Big[ [\sigma_{i_0} \leq \cdots \leq \sigma_{i_l}] \leq 
\cdots \leq [\sigma_{j_0} \leq \cdots \leq \sigma_{j_m}] \leq \cdots <
[\sigma_{k_0} < \cdots < \sigma_{k_n}] \Big] \in \Kstrat^2$$
such that $\st\; \tau \subseteq \st\; \sigma$.
If this is not the case, subdivide $\Lstrat$ until this is true.
Note that there may be many $\sigma$ satisfying this relation.
In this case, choose the unique top dimensional simplex~$\sigma$.
We have the following inclusions:
$$\st\; \tau \subseteq \st\; \sigma \subseteq
\st\; \big[ [\sigma_{i_0} \leq \cdots \leq \sigma_{i_l}] \big] \subseteq 
\st\; \big[ [\sigma_{i_0} ] \big]
	\subseteq \st\; \sigma_{i_0}.$$
Choose an $(\Sstrat, \Jcontrol)$-basic open $U \subseteq M$ containing 
$\st \; \sigma_{i_0}$ such that both open sets are associated to a common
stratum in $\Sstrat$.
Choose an $(\Sstrat', \Jcontrol')$-basic open $V \subseteq M$ contained in 
$\st \; \tau$ such that both open sets are associated to a common
stratum in $\Sstrat'$.
The above inclusions imply an inclusion $i : g^{-1}(V) \to f^{-1}(U)$.
Recall $\Sigma \big( \big[ [\sigma_{i_0} ] \big] \big) = [\sigma_{i_0}]$.
Thus we have the following commutative diagram of solid arrows:
\begin{equation*}
\xymatrix{
\Hgroup_{\ast+m} \big( X, X - f^{-1}  ( U ) \big) 
\ar[r]^-\cong \ar[d]^{i_{\ast+m}} &
\sheaf{F}_{\ast + m} ( \st\; \sigma_{i_0} ) \ar[r]^-{\cong} &
\Sigma^\star \sheaf{F}  \big( \st\; \big[ [ \sigma_{i_0} ] \big] \big ) 
\ar@{-->}[d]^{\ubar{\alpha}( \st\; \tau )} \\
\Hgroup_{\ast+m} \Big( X, X - g^{-1} ( V ) \Big) 
\ar[rr]^-\cong \ar[d]^{\frown}  &&
\sheaf{G}_{\ast + m} (\st\; \tau ) \ar[d]^{\Gfunc_\ast ( \st\; \tau ) } \\
\Hgroup_{\ast} \Big( g^{-1} ( V ) \Big) \ar[d]^{i_\ast}
&& \ar[ll]^-\cong  \cosheaf{G}_{\ast} ( \st\; \tau )  
\ar@{-->}[d]^{\lbar{\alpha} ( \st\; \tau ) }\\
\Hgroup_{\ast} \big( f^{-1} ( U ) \big)
&
\ar[l]^-\cong  \cosheaf{F}_{\ast} \big( \phi ( \st\; \sigma_{i_0} ) \big) & \ar[l]^-{\cong} 
\Sigma^\star \cosheaf{F} \Big( \st\; \big[ [ \sigma_{i_0} ] \big] \big ) \Big) 
}
\end{equation*}
The bisheaf map $\Sigma^\star \bisheaf{F}_\ast \to \bisheaf{G}_\ast$
is generated by defining, for each $\tau \in \Lstrat$, the unique maps
$\ubar{\alpha} ( \st\; \tau )$ and $\lbar{\alpha}( \st\; \tau )$
that make the above diagram commute.
\end{proof}

\begin{corr}
\label{corr:main}
Every map $f \in \mathfrak{W}(X, M)$ has an open neighborhood $\mathcal{U} \subseteq \mathfrak{W}(X, M)$
such that for each map $g \in \mathcal{U}$ their isobisheaf stacks
$\Fstack_\ast$ and $\Gstack_\ast$ are related by canonical stack maps 
in~$\Stack(M)$:
$$\Fstack_\ast \leftarrow \Sigma^\star \Fstack_\ast \to \Gstack_\ast .$$
\end{corr}
\begin{proof}
A bisheaf map gives rise to a canonical map of isobisheaf stacks as constructed
in Example \ref{ex:stack_example}.
The two stack maps follow from the two bisheaf maps of Theorem~\ref{thm:main}.
\end{proof}

We now discuss how Theorem \ref{thm:main} and Corollary \ref{corr:main} imply the property of \emph{Stability}
for persistent local systems mentioned in Section \ref{sec:introduction}.
For every map $f \in \mathfrak{W}(X,M)$, there is a triangulation $(M, \Kstrat)$
such that $f$ is $\Kstrat$-constructible.
Choose a metric on $M$ and recall $\rho > 0$ of Equation \ref{eq:injectivity}.
Note that $\rho$ is a measure of the coarseness of~$\Kstrat$.
The finer the triangulation, the smaller $\rho$ gets.
Given any any \'etale
open $a : A \to M$, the metric on $M$ lifts to a metric on $A$.
Define the distance between $a$ and its shrinking $\dot a$ as the Hausdorff
distance between $A$ and $\dot A$ along the inclusion $\dot A \mono A$.
The Hausdorff distance between $a$ and $\dot a$ is at most $\rho$.
Let $\Fstack_\ast$ be the isobisheaf stack associated to $f$ and let $\Gstack_\ast$
be the isobisheaf stack associated to any $g \in \mathfrak{W}(X,M)$ such that $\sup_{x \in X} d ( fx, gx ) < \rho$.
By Corollary \ref{corr:main} and Equation \ref{eq:subquotient}, the persistent colocal
system $\image \Ffunc_a$ restricted to $\dot A$ is a subquotient of $\image \Gfunc_{\dot a}$.


\section{Examples}
\label{sec:examples}
We have carefully chosen three examples to illustrate key behaviors of persistent (co)local systems.

\begin{ex}
\label{ex:one}
Let $\Rspace^2$ be the plane parameterized by polar coordinates $(r, \theta)$ and
$\Sstrat$ the stratification of $\Rspace^2$ consisting of the following two strata:
the origin $(0,0)$ is the $0$-stratum and $\Rspace^2 - \{ (0,0) \}$ is the $2$-stratum.
The stratification $\Sstrat$ is a Whitney stratification \cite[\S 5]{Mather} of the plane thus
admitting control data $(\Rspace^2, \Sstrat, \Jcontrol)$ \cite[\S 7]{Mather}.
Let $\Sspace^1$ be the circle parameterized by $[0, 2\pi)$
and let $X := [0, \infty) \times \Sspace^1 \times \Sspace^1$.
Define the map
$f : X \to \Rspace^2$ as $f(r, \phi, \theta) = (r, \theta)$.
The map $f$ is $(\Sstrat, \Jcontrol)$-constructible.

We now examine the bisheaf $\bisheaf{F}_1$ of $f$ in dimension one
as constructed in Example~\ref{ex:bisheaf}.
Let $V \subseteq U \subseteq \Rspace^2$ be two 
$(\Sstrat, \Jcontrol)$-basic opens
where $U$ is associated to the $0$-stratum and $V$ to the $2$-stratum.
Since all of $\Rspace^2$ is an $\Sstrat$-basic open, 
$\bisheaf{F}_1$ is, by \cite[Proposition 4.11]{curry_patel} and \cite[Theorem 6.1]{curry_patel},
 uniquely determined (up to an isomorphism) by the following
commutative diagram:
	\begin{equation*}
	\xymatrix{
    \sheaf{F}_3(U) \cong 0 \ar[rr]^0 \ar[d] && 
    \Zspace \cong \sheaf{F}_3(V) \ar[d]^{\id}\\
    \cosheaf{F}_1(U) \cong \Zspace \oplus \Zspace && 
    \Zspace \cong \cosheaf{F}_1(V). \ar[ll]^-{1 \mapsto (1,0)}
    }
\end{equation*}
The restriction $f |_{f^{-1}(V)} : f^{-1}(V) \to V$ is a fiber bundle.
By Proposition \ref{prop:thom_isomorphism}, $\Ffunc_1(V)$ is an isomorphism. 

Now consider the isobisheaf stack $\Fstack_1$ of the bisheaf $\bisheaf{F}_1$
as constructed in Example~\ref{ex:stack_example}.
For any \'etale open $a : A \to \Rspace^2$ that covers
the origin, $\image \Fstack_1(a) = 0$.
For any \'etale open $b : B \to \Rspace^2$ that avoids the origin,
$\image \Fstack_1(b)$ is the constant colocal system $\Zspace$.


Note that $\Hfunc_1 \big( f^{-1}(0) \big)$ is isomorphic to $\Zspace \oplus \Zspace$
but $\image \Ffunc(U)$ is zero indicating that $\Hfunc_1\big(f^{-1}(0)\big)$ is not stable.
Indeed, we can make an arbitrarily small perturbation to~$f$, with respect to the Euclidean metric on $\Rspace^2$, 
so that the pre-image of the origin is empty.

\end{ex}

\begin{ex}
\label{ex:two}
Let $\Rspace^2$ be the plane parameterized by polar coordinates $(r, \theta)$ and
$\Sstrat$ the stratification of $\Rspace^2$ consisting of the following two strata:
the origin $(0,0)$ is the $0$-stratum and $\Rspace^2 - \{ (0,0) \}$ is the $2$-stratum.
The stratification $\Sstrat$ is a Whitney stratification of the plane thus
admitting control data $(\Rspace^2, \Sstrat, \Jcontrol)$.
Let $\Sspace^1$ be the circle parameterized by $[0, 2\pi)$,
let $X := [0, \infty) \times \Sspace^1 \times \Sspace^1$, and
let $X_0 := \{ 0\} \times \Sspace^1 \times \Sspace^1$.
Define the map
$f : X/X_0 \to \Rspace^2$ as $f(r, \phi, \theta) = (r, \theta)$.
The map $f$ is $(\Sstrat, \Jcontrol)$-constructible.

We now examine the bisheaf $\bisheaf{F}_1$ of $f$ in dimension one
as constructed in Example~\ref{ex:bisheaf}.
Let $V \subseteq U \subseteq \Rspace^2$ be two 
$(\Sstrat, \Jcontrol)$-basic opens
where $U$ is associated to the $0$-stratum and $V$ to the $2$-stratum.
Since all of $\Rspace^2$ is an $\Sstrat$-basic open, 
$\bisheaf{F}_1$ is, by \cite[Proposition 4.11]{curry_patel} and \cite[Theorem 6.1]{curry_patel},
uniquely determined (up to an isomorphism) by the following
commutative diagram:
	\begin{equation*}
	\xymatrix{
    \sheaf{F}_3(U) \cong \Zspace \ar[rr]^\id \ar[d] && 
    \Zspace \cong \sheaf{F}_3(V) \ar[d]^{\id}\\
    \cosheaf{F}_1(U) \cong 0 &&
    \Zspace \cong \cosheaf{F}_1(V). \ar[ll]
    }
	\end{equation*}
The restriction $f |_{f^{-1}(V)} : f^{-1}(V) \to V$ is a fiber bundle.
By Proposition \ref{prop:thom_isomorphism}, $\Ffunc_1(V)$ is an isomorphism. 

Now consider the isobisheaf stack $\Fstack_1$ of the bisheaf $\bisheaf{F}_1$
as constructed in Example~\ref{ex:stack_example}.
For any \'etale open $a : A \to \Rspace^2$ that covers
the origin, $\image \Fstack_1(a) = 0$.
For any \'etale open $b : B \to \Rspace^2$ that avoids the origin,
$\image \Fstack_1(b)$ is the constant persistent colocal system~$\Zspace$.

\end{ex}

\begin{ex} 
Let $X_0$ be the torus parameterized by $[0,2\pi) \times [0, 2 \pi)$
and 
$$D := \{ (r,\theta) \subseteq \Rspace^2 \; | \; r \leq 1 \text{ and } 0 \leq \theta < 2\pi \}$$
the closed disk of radius one.
Once again, we are using polar coordinates to label points in the plane.
Let $x \in X_0$ be the distinguished point $(0,0)$.
Let $A$ and $B$ be two copies of $D$.
Glue the boundary of $A$ to $X_0$ along the map 
$\phi_A : (1,\theta) \to (\theta, 0)$ and
glue the boundary of $B$ to $X_0$ along the map
$\phi_B : (1, \theta) \to (0,\theta)$.
Call the resulting space 
$X := X_0 \cup_{\phi_A} A \cup_{\phi_B} B$.

Let $\Sspace^2 := \Rspace^2 \cup \{ \infty \}$ be the $2$-sphere
with the following stratification.
Let $S_0 \subset \Sspace^2$ be the point $(1,0)$, $S_1 \subset \Sspace^2$ the arc
$\{ (1, \theta) \; | \; 0 < \theta < 2\pi \}$,
$S_2$ the connected component of $\Sspace^2 - S_1$ containing the origin,
and $S_3$ the connected component of $\Sspace^2 - S_1$ containing infinity.
The poset $\Sstrat := \{ S_0, S_1, S_2, S_3 \}$ is a 
Whitney stratification of $\Sspace^2$ thus admitting control data 
$(\Sspace^2, \Sstrat, \Jcontrol)$.
Finally, define $f : X \to \Sspace^2$ as the $(\Sstrat, \Jcontrol)$-constructible 
map that takes $x$ to $S_0$, the interior of $A$ homeomorphically to $S_2$, the interior of $B$ 
homeomorphically to $S_3$, and the torus
$X_0$ to the circle $S_0 \cup S_1$ as shown in Figure \ref{fig:example_3}.
	\begin{figure}
	\centering
	\includegraphics{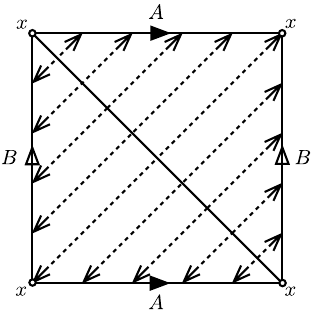}
	\caption{Here we have an illustration of the torus $X_0$ as a square with opposite sides glued.
	The boundary of the disk $A$ is glued to the torus along the horizontal circle and the boundary of $B$ is glued
	to the torus along the vertical circle as indicated.
	The map $f$ takes $x$ to $X_0$ and takes the two circles $(0, \theta)$ and $(\theta, 0)$
	around the circle $S_0 \cup S_1$.
	For a point $p$ in the rest of the torus $X_0$, project~$p$ away from the diagonal to one of the circles
	and apply $f$.
	}
	\label{fig:example_3}
	\end{figure}

Now consider the bisheaf $\bisheaf{F}_0$ of $f$ in dimension zero
as constructed in Example \ref{ex:bisheaf}.
Let $U_0 \subseteq \Sspace^2$ be an $(\Sstrat, \Jcontrol)$-basic open associated to the stratum $S_0$,
$U_1 \subseteq U_0$ an $(\Sstrat, \Jcontrol)$-basic open associated
to the stratum $S_1$, $U_2 \subseteq U_1$ an $(\Sstrat, \Jcontrol)$-basic open
associated to the stratum $S_2$, and $U_3 \subseteq U_1$ an $(\Sstrat, \Jcontrol)$-basic
open associated to the stratum $S_3$.
All four strata are contractible and therefore 
$\bisheaf{F}_0$ is uniquely determined (up to an isomorphism) by
the following commutative diagram:
	\begin{equation*}
    \xymatrix{
    && \sheaf{F}_2 (U_0) \cong \Zspace \oplus \Zspace \ar[lld]^{(1,0)} \ar[d]^{\id} \ar[rrd]^{(1,0)} && \\
    \sheaf{F}_2(U_2) \cong \Zspace \ar[d]^{\id} 
    && \sheaf{F}_2(U_1) \cong \Zspace \oplus \Zspace \ar[d]^{ (1,0) } \ar[ll]^{(1,0)} 
    \ar[rr]^{ (1, 0)}
    && \sheaf{F}_2(U_3) \cong \Zspace \ar[d]^{\id}  \\
    \cosheaf{F}_0(U_2) \cong \Zspace \ar[rr]^{\id} \ar[rrd]^{\id} && 
    \cosheaf{F}_0(U_1) \cong \Zspace \ar[d]^{\id}
    && \cosheaf{F}_0(U_3) \cong \Zspace \ar[ll]^{\id}  \ar[lld]^{\id} \\
    && \cosheaf{F}_0(U_0) \cong \Zspace. &&
    }
    \end{equation*}
The restrictions $f |_{f^{-1}(U_2)} : f^{-1}(U_2) \to U_2$ and $f |_{f^{-1}(U_3)} : f^{-1}(U_3) \to U_3$
are fiber bundles.
By Proposition \ref{prop:thom_isomorphism}, $\Ffunc_0(U_2)$ and $\Ffunc_0(U_3)$ are isomorphisms.
Let $\Fstack_0$ be the isobisheaf stack of $\bisheaf{F}_0$
as constructed in Example~\ref{ex:stack_example}.
For any \'etale open $a : A \to \Sspace^2$, $\image \Fstack_0(a)$
is the constant persistent colocal system $\Zspace$.

We now construct a second constructible map $h : X \to \Sspace^2$.
Let $\Sstrat'$ be the stratification on $\Sspace^2$ consisting of the origin 
as the $0$-stratum $S'_1$ and $\Sspace^2 - S_1'$ as the 2-stratum $S_2'$.
Once again, $(\Sspace^2, \Sstrat')$ is a Whitney stratification and
therefore admits control data $(\Sspace^2, \Sstrat', \Jcontrol')$.
Define $h$ as the map that takes the interior of $B$ homeomorphically 
to $S_2'$ and the rest of $X$ to the origin $S_1'$.
Now consider the bisheaf $\bisheaf{H}_0$ of $h$ in dimension $0$
as constructed in Example \ref{ex:bisheaf}.
Let $U \subseteq \Sspace^2$ be an $\Sstrat$-basic open associated
to $S_1'$ and $V \subseteq U$ an $\Sstrat$-basic open associated
to $S_2'$.
Then $\bisheaf{H}_0$ is uniquely determined (up to an isomorphism) by
the following commutative diagram:
\begin{equation*}
    \xymatrix{
    \sheaf{H}_2(U) \cong \Zspace \ar[rr]^0 \ar[d]^{0} && 
    \sheaf{H}_2(V) \cong \Zspace \ar[d]^{\id} \\
    \cosheaf{H}_0(U) \cong \Zspace  && 
    \cosheaf{H}_0(V) \cong \Zspace.
    \ar[ll]^{\id}
    }
    \end{equation*}
The support of a class in $\sheaf{H}_2(U)$ is the torus $X_0$
whereas the support of a class in $\sheaf{H}_2(V)$ is in the interior of $B$.
Thus the top horizontal map is zero.
The left vertical map is the cap product with the pullback of the orientation
to $\cosheaf{H}^2(U)$.
This pullback is zero making the cap product zero.
Let $\Hstack_0$ be the isobisheaf stack of $\bisheaf{H}_0$
as constructed in Example~\ref{ex:stack_example}.
For any \'etale open $a : A \to \Sspace^2$ that covers the origin,
$\image \Hstack_0(a)$ is zero because the cap product over $U$ is zero.
This zero is explained by the fact that we may perturb $h$
by an arbitrarily small amount, with respect to any metric on $\Sspace^2$
as in the proof of Theorem \ref{thm:main}, so that the pre-image
of the origin is empty.

By Theorem \ref{thm:main}, $f$ has an open neighborhood $W \subseteq \mathfrak{W}(X, \Sspace^2)$ 
such that for each map $g \in W$ their bisheaves 
are related by canonical bisheaf maps 
$$\bisheaf{F}_0 \leftarrow \Sigma^\star \bisheaf{F}_0 \to \bisheaf{G}_0.$$
By Corollary \ref{corr:main}, their isobisheaf stacks are related by
canonical stack maps
$$\Fstack_0 \leftarrow \Sigma^\star \Fstack_0 \to \Gstack_0.$$
Consider the \'etale open $\id : \Sspace^2 \to \Sspace^2$.
The shrinking of $\id$ is $\id$ itself.
This means that the persistent colocal system
$\image \Fstack_0(\id)$ is a subquotient of $\image \Gstack_0(\id)$.
However, $\image \Hstack_0(\id) = 0$ and 
therefore $h$ cannot be in the open set $W$.
Our stability theorem is inherently local.
\end{ex}


\bibliographystyle{acm}
\bibliography{main}

\appendix

\newpage

\section{Sheafification}
\label{sec:sheafification}

Fix a stratified space $(X, \Sstrat)$ as in Definition \ref{defn:stratifed_space}.
The poset $\Basic(X, \Sstrat)$ of all $\Sstrat$-basic opens is a basis for the topology on $X$.
Choose a sub-basis $\Acat \subseteq \Basic(X, \Sstrat)$ and consider a
contravariant functor $\Ffunc :  \Acat \to \Ab$ such that for every pair $J \subseteq I$
of $\Sstrat$-basic opens associated to a common stratum, the map $\Ffunc(J \subseteq I)$
is an isomorphism.
We call such a contravariant functor $\Ffunc$ an \emph{$\Sstrat$-constructible contravariant functor}.
Given such an~$\Ffunc$, we may extend it to an $\Sstrat$-constructible sheaf $\sheaf{F}$ over $X$
as follows.
For any open set $U \subseteq X$, denote by $\Acat(U) \subseteq \Acat$ be the subposet consisting of all open
sets contained in $U$.
Let $\sheaf{F}(U) := \lim \Ffunc |_{\Acat(U)}$.
For every pair of open sets $V \subseteq U$, the inclusion $\Acat(V) \subseteq \Acat(U)$ induces a 
canonical map $\sheaf{F}(V \subseteq U)$ between the two limits.
In this section, we use the equivalence between the category of $\Sstrat$-constructible sheaves
$\Sheaf(X, \Sstrat)$ and the category of functors $\big[ \Exit(X, \Sstrat), \Ab \big]$
from the exit path category $\Exit(X,\Sstrat)$ to show that $\sheaf{F}$ is an 
$\Sstrat$-constructible sheaf as in see Definition~\ref{defn:constructible_sheaf}.

An \emph{exit path} in $(X, \Sstrat)$ is a continuous map $\gamma : [0,1] \to X$ such that the dimension
of the stratum containing the point $\gamma(t)$ is non-decreasing with increasing~$t$.
Two exit paths $\alpha , \beta : [0,1] \to X$, where $\alpha(0) = \beta(0)$ and $\alpha(1) = \beta(1)$,
are \emph{equivalent} if, roughly speaking, there is a homotopy of exit paths taking $\alpha$ to $\beta$.
See \cite[Definition 4.5]{curry_patel} for a precise definition of equivalence in the dual setting of entrance paths, 
\cite[Definition 7.4]{treumann} for equivalence in the $2$-category setting, 
and \cite[Definition A.6.2]{higher_algebra} for equivalence in the $\infty$-category setting.
The \emph{exit path category} $\Exit(X, \Sstrat)$ consists of points of $X$ as objects 
and a morphism from $x$ to $x'$ is an equivalence class of exit paths starting at $x$ and ending at~$x'$.

The category $\Sheaf(X, \Sstrat)$ of $\Sstrat$-constructible sheaves over $X$ is equivalent to the 
category of functors $\big [ \Exit(X, \Sstrat), \Ab \big]$.
See \cite[Theorem 6.1]{curry_patel} for 
the equivalence in the dual setting of constructible cosheaves, \cite[Theorem 7.14]{treumann} for the $2$-category setting, 
and \cite[Theorem A.9.3]{higher_algebra} for the $\infty$-category setting.
There are two functors 
	\begin{equation*}
	\begin{tikzcd}
	\Sheaf(X, \Sstrat) \ar[rr, bend left, "\Phi"] && \big [ \Exit(X, \Sstrat), \Ab \big] \ar[ll, "\Psi", bend left]
	\end{tikzcd}
	\end{equation*}in this equivalence, but we only
need $\Psi$, which we now describe on an object $\Hfunc : \Exit(X, \Sstrat) \to \Ab$.
For an open set $U \subseteq X$, denote by $\Exit(U, \Sstrat |_U)$ the full subcategory
of $\Exit(X, \Sstrat)$ restricted to the stratified space $(U, \Sstrat |_U)$.
Then $\Psi(\Hfunc)(U) := \lim_{\Exit(U, \Sstrat |_U)} \Hfunc$.
For every pair of open sets $V \subseteq U$, $\Psi(\Hfunc)(V \subseteq U)$ is the universal
homomorphism between the two limits.

\begin{prop}
Let $\Acat \subseteq \Basic(X, \Sstrat)$ be a sub-basis and $\Ffunc : \Acat \to \Ab$ 
an $\Sstrat$-constructible contravariant functor.
Then the contravariant functor $\sheaf{F} : \Open(X) \to \Ab$, as constructed above from $\Ffunc$,
is an $\Sstrat$-constructible sheaf.
\end{prop}

\begin{proof}
First, we construct a functor $\Hfunc : \Exit(X, \Sstrat) \to \Ab$ from $\Ffunc$.
Next, we show that the two contravariant functors $\Psi(\Hfunc), \sheaf{F} : \Open(X) \to \Ab$ 
are naturally isomorphic thus proving $\sheaf{F}$ is an $\Sstrat$-constructible sheaf. 

For a point $x \in X$, let $\Hfunc(x)$ be the colimit of $\Ffunc$ over all open sets containing~$x$.
For an open set $I \in \Acat$ containing $x$, denote by $\Hfunc(x \in I) : \Ffunc(I) \to \Hfunc(x)$
the canonical homomorphism to the colimit.
Note that if $I$ is associated to the stratum containing~$x$, then $\Hfunc(x \in I)$ is an isomorphism.
Now consider an exit path $\alpha$ in $(X, \Sstrat)$ that meets at most two strata
and satisfies the following condition.
Suppose the point $\alpha(0)$ lies on a stratum $S_1 \in \Sstrat$ and the point $\alpha(1)$
lies on a stratum $S_2 \in \Sstrat$.
Then we require that there is an $\Sstrat$-basic open $I \in \Acat$ associated to $S_1$ containing
the path~$\alpha$.
Define $\Hfunc(\alpha) : \Hfunc(\alpha(0)) \to \Hfunc(\alpha(1))$ 
as the following composition:
	\begin{equation*}
	\begin{tikzcd}
	\Hfunc(\alpha(0)) \ar[rr, "\Hfunc( \alpha(0) \in I )^{-1}", "\cong"']
	&& \Ffunc(I) \ar[rr, "\Hfunc(\alpha(1) \in I)"]
	&& \Hfunc(\alpha(1)).
	\end{tikzcd}
	\end{equation*}
The homomorphism $\Hfunc(\alpha)$ is independent of the choice of $I$; see proof of
\cite[Theorem 6.1]{curry_patel}.
An arbitrary exit path $\beta$ in $(X, \Sstrat)$ can be written as a composition of simpler exit paths
each satisfying the conditions imposed on $\alpha$ above.
The homomorphism $\Hfunc(\beta)$ is simply the composition of homomorphisms
associated to each of these simpler exit paths.
Furthermore, if $\beta'$ is a second exit path equivalent to $\beta$, then
$\Hfunc(\beta) = \Hfunc(\beta')$; see proof of \cite[Theorem 6.1]{curry_patel}.

We now build a natural isomorphism $\mu : \Psi(\Hfunc) \Rightarrow  \sheaf{F}$.
Suppose $I \in \Acat$ is associated to a stratum $S \in \Sstrat$.
Then, by \cite[Proposition 4.11]{curry_patel}, every point $x \in I \cap S$ is initial in $\Exit(I , \Sstrat |_I)$.
This implies $\Psi(\Hfunc)(I)$ is canonically isomorphic to $\Hfunc(x)$, which is canonically
isomorphic to $\Ffunc(I)$, which is canonically isomorphic to $\sheaf{F}(I)$.
Let $\mu(I) : \Psi(\Hfunc)(I) \to \sheaf{F}(I)$ be this canonical isomorphism.
For every pair $J \subseteq I$ in $\Acat$, there is, by \cite[Proposition 4.11]{curry_patel}, 
a unique exit path~$\alpha$ in $I$ from $x \in I \cap S$ to
$y \in J \cap S'$, where $S'$ is the stratum associated to $J$.
Combined with the fact that both $x$ and $y$ are initial in $\Exit(I, \Sstrat |_I)$ and $\Exit(J, \Sstrat |_J)$,
respectively, the following diagram commutes:
	\begin{equation*}
	\begin{tikzcd}
	\Psi(\Hfunc)(I)  \ar[d, "\mu(I)", "\cong"'] \ar[rr, "\Psi(\Hfunc)(J \subseteq I)"] && 
	\Psi(\Hfunc)(J) \ar[d, "\mu(J)", , "\cong"'] \\
	\sheaf{F}(I) \ar[rr, "\sheaf{F}(J \subseteq I)"] && \sheaf{F}(J).
	\end{tikzcd}
	\end{equation*}
Now let $U \subseteq X$ be an arbitrary open set.
Since $\Psi(\Hfunc)$ satisfies the gluing axiom and $\Acat(U)$ is an open cover of $U$, 
the universal homomorphism
$\Psi(\Hfunc)(U) \to \lim \Psi(\Hfunc) |_{\Acat(U)}$ is an isomorphism.
Furthermore, there is a canonical isomorphism $\lim \Psi(\Hfunc) |_{\Acat(U)} \to \sheaf{F}(U)$
because both groups are limits over naturally isomorphic diagrams $\Psi(\Hfunc) |_{\Acat(U)}$ and
$\sheaf{F} |_{\Acat(U)}$, respectively.
Thus for every pair of open sets $V \subseteq U$, the following diagram commutes:
	\begin{equation*}
	\begin{tikzcd}
	\Psi(\Hfunc)(U)  \ar[d, "\mu(U)", , "\cong"'] \ar[rr, "\Psi(\Hfunc)(V \subseteq U)"] && 
	\Psi(\Hfunc)(V) \ar[d, "\mu(V)", , "\cong"'] \\
	\sheaf{F}(U) \ar[rr, "\sheaf{F}(V \subseteq U)"] && \sheaf{F}(V).
	\end{tikzcd}
	\end{equation*}
Therefore, $\sheaf{F}$ is canonically isomorphic to $\Psi(\Hfunc)$.

\end{proof}

\newpage
\section{Cosheafification}
\label{sec:cosheafification}

Fix a stratified space $(X, \Sstrat)$ as in Definition \ref{defn:stratifed_space}.
The poset $\Basic(X, \Sstrat)$ of all $\Sstrat$-basic opens is a basis for the topology on $X$.
Choose a sub-basis $\Acat \subseteq \Basic(X, \Sstrat)$ and consider a
covariant functor $\Ffunc :  \Acat \to \Ab$ such that for every pair $J \subseteq I$
of $\Sstrat$-basic opens associated to a common stratum, the map $\Ffunc(J \subseteq I)$
is an isomorphism.
We call such a covariant functor $\Ffunc$ an \emph{$\Sstrat$-constructible covariant functor}.
Given such an~$\Ffunc$, we may extend it to an $\Sstrat$-constructible cosheaf $\cosheaf{F}$ under $X$
as follows.
For any open set $U \subseteq X$, denote by $\Acat(U) \subseteq \Acat$ be the subposet consisting of all open
sets contained in $U$.
Let $\cosheaf{F}(U) := \colim \Ffunc |_{\Acat(U)}$.
For every pair of open sets $V \subseteq U$, the inclusion $\Acat(V) \subseteq \Acat(U)$ induces a 
canonical map $\cosheaf{F}(V \subseteq U)$ between the two colimits.
In this section, we use the equivalence between the category of $\Sstrat$-constructible cosheaves
$\Cosheaf(X, \Sstrat)$ and the category of functors $\big[ \Ent(X, \Sstrat), \Ab \big]$
from the entrance path category $\Ent(X,\Sstrat)$ to show that $\cosheaf{F}$ is an 
$\Sstrat$-constructible cosheaf as in Definition~\ref{defn:constructible_sheaf}.
The entrance path category $\Ent(X, \Sstrat)$ is simply the opposite of the exit path
category $\Exit(X, \Sstrat)$ from Appendix \ref{sec:sheafification}.

The category $\Cosheaf(X, \Sstrat)$ of $\Sstrat$-constructible cosheaves under $X$ is equivalent to the 
category of functors $\big [ \Ent(X, \Sstrat), \Ab \big]$;
see \cite[Theorem 6.1]{curry_patel}.
There are two functors 
	\begin{equation*}
	\begin{tikzcd}
	\Cosheaf(X, \Sstrat) \ar[rr, bend left, "\Phi"] && \big [ \Ent(X, \Sstrat), \Ab \big] \ar[ll, "\Psi", bend left]
	\end{tikzcd}
	\end{equation*}
in this equivalence, but we only
need $\Psi$, which we now describe on an object $\Hfunc : \Ent(X, \Sstrat) \to \Ab$.
For an open set $U \subseteq X$, denote by $\Ent(U, \Sstrat |_U)$ the full subcategory
of $\Ent(X, \Sstrat)$ restricted to the stratified space $(U, \Sstrat |_U)$.
Then $\Psi(\Hfunc)(U) := \colim_{\Ent(U, \Sstrat |_U)} \Hfunc$.
For every pair of open sets $V \subseteq U$, $\Psi(\Hfunc)(V \subseteq U)$ is the universal
homomorphism between the two colimits.

\begin{prop}
Let $\Acat \subseteq \Basic(X, \Sstrat)$ be a sub-basis and $\Ffunc : \Acat \to \Ab$ 
an $\Sstrat$-constructible covariant functor.
Then the covariant functor $\cosheaf{F} : \Open(X) \to \Ab$, as constructed above from $\Ffunc$,
is an $\Sstrat$-constructible cosheaf.
\end{prop}

\begin{proof}
First, we construct a functor $\Hfunc : \Ent(X, \Sstrat) \to \Ab$ from $\Ffunc$.
Next, we show that the two covariant functors $\Psi(\Hfunc), \cosheaf{F} : \Open(X) \to \Ab$ 
are naturally isomorphic thus proving $\cosheaf{F}$ is an $\Sstrat$-constructible cosheaf.

We now construct a functor $\Hfunc : \Ent(X, \Sstrat) \to \Ab$ from the covariant 
functor~$\Ffunc$ given above.
For a point $x \in X$, let $\Hfunc(x)$ be the limit of $\Ffunc$ over all open sets containing~$x$.
For an open set $I \in \Acat$ containing $x$, denote by $\Hfunc(x \in I) : \Hfunc(x) \to \Ffunc(I)$
the canonical homomorphism from the limit.
Note that if $I$ is associated to the stratum containing~$x$, then $\Hfunc(x \in I)$ is an isomorphism.
Now consider an entrance path $\alpha$ in $(X, \Sstrat)$ that meets at most two strata
and satisfies the following condition.
Suppose the point $\alpha(0)$ lies on a stratum $S_2 \in \Sstrat$ and the point $\alpha(1)$
lies on a stratum $S_1 \in \Sstrat$.
Then we require that there is an $\Sstrat$-basic open $I \in \Acat$ associated to $S_1$ containing
the path~$\alpha$.
Define $\Hfunc(\alpha) : \Hfunc(\alpha(0)) \to \Hfunc(\alpha(1))$ 
as the following composition
	\begin{equation*}
	\begin{tikzcd}
	\Hfunc(\alpha(0)) \ar[rr, "\Hfunc( \alpha(0) \in I )"]
	&& \Ffunc(I) \ar[rr, "\Hfunc(\alpha(1) \in I)^{-1}", "\cong"']
	&& \Hfunc(\alpha(1)).
	\end{tikzcd}
	\end{equation*}
The homomorphism $\Hfunc(\alpha)$ is independent of the choice of $I$; see proof of
\cite[Theorem 6.1]{curry_patel}.
An arbitrary entrance path $\beta$ in $(X, \Sstrat)$ can be written as a composition of simpler entrance paths
each satisfying the conditions imposed on $\alpha$ above.
The homomorphism $\Hfunc(\beta)$ is simply the composition of homomorphisms
associated to each of these simpler entrance paths.
Furthermore, if $\beta'$ is a second entrance path equivalent to $\beta$, then
$\Hfunc(\beta) = \Hfunc(\beta')$; see proof of \cite[Theorem 6.1]{curry_patel}.

We now build a natural isomorphism $\mu : \cosheaf{F} \Rightarrow \Psi(\Hfunc)$.
Suppose $I \in \Acat$ is associated to a stratum $S \in \Sstrat$.
Then, by \cite[Proposition 4.11]{curry_patel}, every point $x \in I \cap S$ is final in $\Ent(I , \Sstrat |_I)$.
This implies $\Psi(\Hfunc)(I)$ is canonically isomorphic to $\Hfunc(x)$, which is canonically
isomorphic to $\Ffunc(I)$, which is canonically isomorphic to $\cosheaf{F}(I)$. 
Let $\mu(I) : \cosheaf{F}(I) \to \Psi(\Hfunc)(I)$ be this canonical isomorphism.
For every pair $J \subseteq I$ in $\Acat$, there is, by \cite[Proposition 4.11]{curry_patel}, 
a unique entrance path~$\alpha$ in $I$ from $y \in J \cap S'$ to
$x \in I \cap S$, where $S'$ is the stratum associated to $J$.
Combined with the fact that both $x$ and $y$ are final in $\Ent(I, \Sstrat |_I)$ and $\Ent(J, \Sstrat |_J)$,
respectively, the following diagram commutes:
	\begin{equation*}
	\begin{tikzcd}
	\cosheaf{F}(J)  \ar[d, "\mu(J)", "\cong"']  \ar[rr, "\cosheaf{F}(J \subseteq I)"] && \cosheaf{F}(I)  
	\ar[d, "\mu(I)", , "\cong"'] \\
	\Psi(\Hfunc)(J) \ar[rr, "\Psi(\Hfunc)(J \subseteq I)"] && \Psi(\Hfunc)(I).
	\end{tikzcd}
	\end{equation*}
Now let $U \subseteq X$ be an arbitrary open set.
Since $\Psi(\Hfunc)$ satisfies the gluing axiom and $\Acat(U)$ is an open cover of $U$, 
the universal homomorphism
$\colim \Psi(\Hfunc) |_{\Acat(U)} \to \Psi(\Hfunc)(U)$ is an isomorphism.
Furthermore, there is a canonical isomorphism $\colim \Psi(\Hfunc) |_{\Acat(U)} \to \cosheaf{F}(U)$
because both groups are colimits over naturally isomorphic diagrams $\Psi(\Hfunc) |_{\Acat(U)}$ and
$\cosheaf{F} |_{\Acat(U)}$, respectively.
Thus for every pair of open sets $V \subseteq U$, the following diagram commutes:
	\begin{equation*}
	\begin{tikzcd}
	\cosheaf{F}(V)  \ar[d, "\mu(V)", , "\cong"']  \ar[rr, "\cosheaf{F}(V \subseteq U)"] && \cosheaf{F}(U)  
	\ar[d, "\mu(U)", , "\cong"'] \\
	\Psi(\Hfunc)(V) \ar[rr, "\Psi(\Hfunc)(V \subseteq U)"] && \Psi(\Hfunc)(U).
	\end{tikzcd}
	\end{equation*}
Therefore, $\cosheaf{F}$ is canonically isomorphic to $\Psi(\Hfunc)$.

\end{proof}

\end{document}